\documentclass[a4paper,11pt]{article}
\usepackage[utf8]{inputenc} 
\usepackage[T1]{fontenc}    
\usepackage[english]{babel} 
\usepackage{a4wide}         
\usepackage{xcolor}

\usepackage{amsmath, amssymb, amsthm, amsfonts , stmaryrd} 
\usepackage{mathrsfs}  
\usepackage{bbold}     
\usepackage{bigints}   
\usepackage{upgreek}   
\usepackage{systeme}   
\usepackage{calc}      

\usepackage{graphicx}
\graphicspath{{./images/}} 
\usepackage[font=small,labelfont=bf]{caption} 
\usepackage{subcaption} 

\usepackage{hyperref}

\newtheoremstyle{custom_theorem} 
  {10pt} 
  {10pt} 
  {\itshape} 
  {} 
  {\bfseries} 
  {.} 
  { } 
  {} 

\newtheoremstyle{custom_remark} 
  {10pt} 
  {10pt} 
  {\normalfont} 
  {} 
  {\bfseries} 
  {:} 
  { } 
  {} 

\theoremstyle{custom_theorem}
\newtheorem{theorem}{Theorem}[section]
\newtheorem{proposition}[theorem]{Proposition}
\newtheorem{lemma}[theorem]{Lemma}
\newtheorem{corollary}[theorem]{Corollary}
\newtheorem{definition}[theorem]{Definition}

\theoremstyle{custom_remark}
\newtheorem{remark}{Remark}[section]

\numberwithin{equation}{section}

\renewcommand{\P}{\mathbb{P}}

\newcommand{\E}{\mathbb{E}}
\newcommand{\N}{\mathbb{N}}
\newcommand{\1}{\mathbb{1}}
\newcommand{\C}{\mathbb{C}}

\newcommand{\R}{\mathbb{R}}
\newcommand{\K}{\mathbb{K}}

\newcommand{\dd}{\mathrm{d}}

\newcommand{\HH}{\mathfrak{h}}
\newcommand{\hh}{\mathfrak{h}}
\newcommand{\cc}{\mathfrak{c}}

\DeclareUnicodeCharacter{2212}{\ensuremath{-}}

\title{Spectral aspects of random heavy-tailed tensors}
\author{Rémi Bonnin\thanks{Aix-Marseille Univ, CNRS, I2M, Marseille, France. Email: \href{mailto:remi.bonnin@ens.fr}{remi.bonnin@ens.fr}}\,, Alexis Imbert 
\thanks{Universit\'{e} de Bordeaux, Institut de Math\'{e}matiques de Bordeaux, 33400 Talence, France. Email: \href{mailto:alexis.imbert@math.u-bordeaux.fr}{alexis.imbert@math.u-bordeaux.fr}}}

\begin{document}
\maketitle
\begin{abstract}

We introduce heavy-Wigner random tensors and establish the convergence of a generating family of invariant linear forms in their entries together with a central limit theorem. The relevant group actions arise naturally in the study of tensor eigenvalues.

Then, we compute the limiting spectral measure of the matrix obtained by contraction of a heavy-Wigner tensor and we derive two applications. The first one concerns the adjacency tensor of a Erd\H{o}s-Rényi hypergraph. We identify the local weak limit of the hypergraph and study the limiting spectral distribution of the matrix contraction. The second application concerns Lévy tensors whose entries belong to the domain of attraction of an $\alpha$-stable law.
\end{abstract}
\tableofcontents

\section{Introduction and main results}

\subsection{General framework}

\paragraph{Symmetric tensor.}
For any $N \geq 1$, let $(e_1,\ldots,e_N)$ denote the canonical basis of $\R^N$ or $\C^N$.
A (real or complex) tensor of order $p \geq 1$ and dimension $N \geq 1$ is a multidimensional array of (real or complex) numbers $T= (T_{i_1,\ldots,i_p})_{1\leq i_1,\ldots,i_p \leq N} \in (\K^{N})^{\otimes p}$ with $\K=\R$ or $\C$. 
The symmetric group $\mathfrak{S}_p$ acts on the space of tensors by permutation of the indices. For $\sigma \in \mathfrak{S}_p$ and $T$ a tensor, we denote $T^\sigma$ the tensor with entries 
$$ T^{\sigma}_{i_1,\ldots,i_p} = T_{i_{\sigma(1)},\ldots,i_{\sigma(p)}} . $$
The tensor $T$ is called {\em symmetric} 
if for all $\sigma \in \mathfrak{S}_p$, $T^\sigma = T$. Matrices naturally act on the space of symmetric tensors as follows. Let $T$ be a symmetric tensor of order $p$ and dimension $N$ and let $U$ be a $N\times N$ matrix. We define $U \cdot T$ as the $p$-order tensor with entries
\[
    (U \cdot T)_{i_1,\ldots,i_p} := \sum_{j_1,\ldots,j_p} T_{j_1,\ldots,j_p} \prod_{k=1}^p U_{j_k i_k} .
\]

\paragraph{Hypergraph.}
As matrices are naturally related to graphs, the same holds for ($p$-order) tensors and ($p$-uniform) hypergraphs.
A {\em hypergraph} is a pair $H=(V,E)$, where $V$ is a finite vertex set and $E$ is a finite multiset of hyperedges; each hyperedge is itself a multiset of vertices. Thus loops and repeated hyperedges are allowed unless explicitly excluded. The order of a hyperedge $e$, denoted by $|e|$, is the number of vertices in $e$, counted with multiplicity. The degree of a vertex $v$ is the number of incidences of $v$ with hyperedges, again counted with multiplicity. A hypergraph is $p$-\emph{uniform} if any hyperedge has order $p$. A hypergraph is simple if every hyperedge has multiplicity one. The \emph{clique expansion} of $H$, denoted by $G_H$, is the graph on the same vertex set obtained by replacing each hyperedge $e$ by all graph edges between pairs of vertices belonging to $e$.

A simple hypergraph is a \emph{hypertree} if it is connected and has no cycles. A hypergraph is a \emph{fat hypertree} if the simple hypergraph obtained by forgetting hyperedge multiplicities is a hypertree. The type of a fat hypertree is the sequence $(q_r)_{r\geq 1}$, where $q_r$ is the number of distinct hyperedges that occur with multiplicity $r$.

\paragraph{Trace invariant.}
 
In \cite{collins_male_gaudreau}, a generating family of $\mathfrak{S}_N$-invariant linear forms is exhibited in the form of test-graphs (see \cite{male_traffic, male_2} for a broader introduction to traffic probability).
The same natural questions can be considered for tensors.
As for matrices, a tensor can therefore be tested on (test-) hypergraphs. Indeed, if $T$ is a $p$-order tensor and $\mathfrak{h}=(V,E)$ is a $p$-uniform hypergraph, we set 
$$ \mathrm{Tr}_{\mathfrak{h}}(T) := \sum_{\phi : V \rightarrow [N]} \prod_{e=\{v_1,\ldots,v_p\} \in E} T_{\phi(v_1),\ldots,\phi(v_p)} .$$
This polynomial in the entries of $T$ is called a {\em trace invariant} associated with $\mathfrak{h}$. Again if $\# \mathfrak{h}$ is the number of connected components of the hypergraph $\mathfrak{h}$, let
$$\mathrm{tr}_{\mathfrak{h}}(T)= \frac{1}{N^{\# \mathfrak{h}}} \mathrm{Tr}_{\mathfrak{h}}(T) $$
denote the normalized trace invariant.
Inspired by traffic theory we should also define what is called {\em injective trace invariant} when the sum now runs on distinct indices. More precisely, if $\mathfrak{h}$ is a hypergraph with $|\mathfrak{h}|=k$ vertices and $\pi$ a partition of these vertices, denote by $\mathfrak{h}^\pi$ the hypergraph where we identify the vertices in the same block of $\pi$, and we can write 

$$ \mathrm{Tr}_{\mathfrak{h}}(T) = \sum_{\pi \in \mathcal{P}(k)} \mathrm{Tr}^0_{\mathfrak{h}^\pi}(T) \quad \mathrm{with} \quad \mathrm{Tr}^0_{\mathfrak{h}^\pi}(T) = \sum_{\genfrac{}{}{0pt}{}{\phi : V(\mathfrak{h}^\pi) \rightarrow [N]}{\mathrm{injective}}} \prod_{e=\{v_1,\ldots,v_p\} \in E} T_{\phi(v_1),\ldots,\phi(v_p)} .$$

Note that one can also express $\mathrm{Tr}^0_{\mathfrak{h}}(T)$ in terms of the  $\mathrm{Tr}_{\mathfrak{h}^\pi}(T)$ using a Möbius inversion argument.

\paragraph{Eigenvalues for tensors.}

Let $p=2$ and $M$ be a tensor of order 2, i.e a matrix. The eigenvalues of $M$ are the $\lambda \in \C$ such that there exists $x \in \C^N \setminus \{0\}$ satisfying
$$ M x = \lambda x .$$
This equation is homogeneous in $x$ and invariant under an orthogonal (or unitary in the complex case) change of basis. The eigenvalues are real if $M$ is real symmetric (or complex Hermitian).
For general $p\geq 3$, there exist two main notions of eigenvalues for tensors. 
These notions were introduced by Qi \cite{QI2005}. In order to define them, we first introduce some operations relevant for the tensor setting. If $T$ and $T'$ are two tensors of order $p$ and $p'$, and $r \leq \mathrm{min}(p,p')$, let $T \star_r T'$ denote the $(p+p'-2r)$-order tensor whose entries are given by the {\em contraction} of the $r$ last legs of $T$ with the $r$ first ones of $T'$, that is 
$$ (T \star_r T')_{i_1,\ldots,i_{p+p'-2r}} =\sum_{1 \leq j_1,\ldots,j_r \leq N} T_{i_1,\ldots,i_{p-r},j_1,\ldots,j_r} T'_{j_1,\ldots,j_r,i_{p-r+1},\ldots,i_{p+p'-2r}} . $$
When $T'=x \in \K^N$ is a vector ($p'=1$), let $T.x:=T \star_1 x$ and more generally for $1 \leq r \leq p$,
$$ T.x^{r} := T \star_r x^{\otimes r} \in (\K^N)^{\otimes (p-r)} . $$
We finally let $x^{\circ r}=(x_1^r,\ldots,x_N^r)^\intercal$ denote the entrywise power. Note that the matrix eigenvalue equation can be written as $M.x= \lambda x$.

The eigenvalues for general tensors are defined as follows. Let $T \in (\K^N)^{\otimes p}$ be a symmetric tensor of order $p$. We say that a complex number $\lambda \in \C$ is an {\em eigenvalue} if there exists $x \in \C^N \setminus \{0\}$ such that
\begin{equation}\label{eq:eigenvalue}
    T.x^{p-1} = \lambda x^{\circ(p-1)} .
\end{equation}
This equation is homogeneous in $x$ but no longer invariant under an orthogonal change of basis whenever $p\geq 3$. These eigenvalues have mainly been studied in the case of hypergraphs, see \cite{mukherjee}. If $\lambda \in \R$ is a real eigenvalue with eigenvector $x \in \R^N \setminus \{0\}$, it is called an H-eigenvalue. 

On the other hand, we say that a complex number $\lambda \in \C$ is an {\em E-eigenvalue} if there exists $x \in \C^N \setminus \{0\}$ such that
\begin{equation}\label{eq:Eeigenvalue}
    T.x^{p-1} = \lambda x \quad \mathrm{and} \quad x^\intercal x =1 .
\end{equation}
This equation is now invariant under an orthogonal change of basis, i.e. $T$ and $U.T$ have the same E-eigenvalues for any $U\in O_N$, but it is not homogeneous in $x$, we therefore impose a additional condition to normalize the eigenvector. The E-eigenvalues $\lambda \in \R$ with eigenvector $x \in \R^N \setminus \{0\}$ are called the Z-eigenvalues. 
This second notion have been studied extensively in the context of data science. Typically, the eigenvector associated with the largest E-eigenvalue is the best rank-$1$ approximation of the tensor by a vector in the $l^2$ ball \cite{Qi} (while the one associated to the largest eigenvalue is the best rank-$1$ approximation of the tensor by a vector in the $l^p$ ball).


\subsection{Main results}
Our first results focus on linear forms invariant under the action of $\mathfrak{S}_N$ or a larger group. We show that the trace invariants associated to $p$-uniform hypergraphs form a complete family of the linear forms invariant by the action of $\mathfrak{S}_N$.
\begin{proposition}\label{thm: formes lineaires invariantes}
    Let $p\geq 2$ and $k\geq 1$ be two integers and let $R:((\C^N)^{\otimes p})^{\otimes k}\rightarrow \C$ be an $\mathfrak{S}_N$-invariant linear form in each $(\C^{N})^{\otimes p}$ component. Then $R$ can be expressed as a linear combination of traces (or injective traces) of $p$-uniform test hypergraphs with $k$ hyperedges.
\end{proposition}
This is a straightforward extension of the matrix version of this proposition proved in \cite{collins_male_gaudreau}. A proof is given in Section \ref{sec:trace_inv}. 

Invariant linear forms are strongly linked to tensor eigenvalues. On the one hand, it is known that $E$-eigenvalues are invariant under an action of $O_N$, i.e. $T$ and $U\cdot T$ have the same eigenvalues for any $U\in O_N$. On the other hand, we show that eigenvalues of a symmetric tensor are invariant under an action of $G_N^p$, which is defined as follows.
\begin{definition}\label{def: G_N^p}
    Let $p\geq 3$ and denote by $\mathbb{U}_p(\mathbb{K})$ the set of the $p$-th roots of unity in $\mathbb{K}$ and $$\mathbb{U}_p(N,\mathbb{K})=\{\mathrm{Diag}(\xi_i,\,i=1,\cdots,N),\,\xi_i\in\mathbb{U}_p(\mathbb{K})\}\cong\mathbb{U}_p(\mathbb{K})^{N}.$$ Define $G_N^p(\mathbb{K}):=\{V_\sigma D,\,\sigma\in\mathfrak{S}_N,\,D\in\mathbb{U}_p(N,\mathbb{K})\}\cong\mathbb{U}_p(\mathbb{K})^{N}\rtimes\mathfrak{S}_N,$ where $V_\sigma$ is the permutation matrix associated to $\sigma$ and $\mathfrak{S}_N$ acts on $\mathbb{U}_p(\mathbb K)^{N}$ by permuting the elements.
\end{definition}

\begin{theorem}\label{prop: eigenvalues invariant by the action}
    Let $p\geq3$, a matrix $U$ is in $G_N^p(\mathbb K)$ if and only if for all symmetric tensor $T\in (\mathbb K^n)^{\otimes p}$ of order $p$, $U\cdot T$ and $T$ have same eigenvalues.
\end{theorem}

From Proposition \ref{thm: formes lineaires invariantes}, we deduce the following corollary which gives a generating family of $G_N^p(\mathbb K)$-invariant linear forms.
\begin{corollary}\label{cor: G_N^p invariant linear forms}
    Let $p\geq 3$ and $k\geq 1$ be two integers and let $R:((\C^N)^{\otimes p})^{\otimes k}\rightarrow \C$ be a $G_N^p(\mathbb K)$-invariant linear form in each $(\C^{N})^{\otimes p}$ component. Then $R$ can be expressed as a linear combination of traces (or injective traces) of $p$-uniform test hypergraphs with $k$ hyperedges such that each vertex has degree divisible by $p$.
\end{corollary}

Note also that the trace invariants associated to $p$-uniform and $2$-regular hypergraphs form a complete family of the polynomials invariant by the action of the orthogonal group $O_N$, see Theorem $3.15$ in \cite{kunisky_moore_wein} or \cite{bonnin_gote} for recent proofs adapted to our notation.

We introduce heavy-Wigner tensors as both a generalization of Wigner tensors and a generalization of heavy-Wigner matrices. This framework is well-suited to studying Erd\H{o}s-Rényi random hypergraphs in various regimes including sparse regimes.

\begin{definition}\label{def: heavy-wigner tensor}
Let $p \geq 2$ be fixed. A sequence of {\em heavy-Wigner tensors} is a sequence $(T^{(N)})_{N \geq 1}$ of symmetric tensors of order $p$ and dimension $N$ with independent entries (up to symmetries) such that for all $i_1,\ldots,i_p$, for all $k\geq 2$, 
$$ \E [T_{i_1,\ldots,i_p}] =: M_1 = 0, \quad  \quad N^{p-1} \E [T_{i_1,\ldots,i_p}^k] = O(1) \text{ as } N \rightarrow \infty , $$
where the $O(1)$ is uniform in $N$ and for all $i_1<\ldots<i_p$ and $k \geq 2$, there exists $M_{k}$ such that when $N \rightarrow \infty$,
$$ N^{p-1} \E [ T_{i_1,\ldots,i_p}^{k}] \rightarrow M_{k} . $$
The sequence $(M_k)$ is the parameter of the heavy-Wigner tensors.
\end{definition}

An example of a heavy-Wigner tensor is given by the (centered and normalized) adjacency tensor of a Erd\H{o}s-Rényi hypergraph where each hyperedge is present with probability $c/N^{p-1}$ for some constant $c>0$. Note that with this normalization, as in the matrix case, on average there are $O(N)$ hyperedges which is what one would expect for a sparse regime. Heavy-Wigner tensors also appear naturally when one truncates random tensors whose entries are heavy-tailed random variables, see Theorem \ref{thm:Lévy}.

\paragraph{Convergence of the trace invariants.}
We establish a convergence result for traace invariants of heavy-Wigner tensors. Proposition \ref{thm: formes lineaires invariantes} shows that it is also a convergence result for any permutation invariant linear form. The following Theorem is a generalization of a result in \cite{male_limit_heavy_wigner} to tensors.

\begin{theorem}\label{thm: trace injective heavy wigner}
    Let $\mathfrak{h}=(V,E)$ be a connected $p$-uniform (test-) hypergraph and $\pi$ be a partition of its vertices. Let $T$ be a heavy-Wigner tensor. Then, we have
    \begin{equation}\label{eq:limite_trace_inj_heavy_wig}
        \lim_{N\rightarrow\infty}\E\left[\mathrm{tr}^{0}_{\mathfrak{h}^\pi}(T)\right]=\1_{\mathfrak{h}^\pi \text{ is a fat hypertree}} \prod_{r \in \N} M_{r}^{q_{r}}.
    \end{equation}
\end{theorem}

This theorem and Corollary \ref{cor:trace_injective_heavy_wig} are proved in Section \ref{sec:conv_trace_inv}. The convergence comes with a central limit theorem: Theorem \ref{thm: CLT for tensors} which is proved in Section \ref{sec:CLT_tensor}.

\paragraph{Contraction of a heavy-Wigner tensor.}

Let $T$ be a heavy-Wigner tensor of order $p$ and dimension $N$ with parameters $(M_r)_{r\geq 2}$ and set $u=\mathbf{1} / \sqrt{N} \in \mathbb{S}^{N-1}$. We consider 
\begin{equation}\label{eq:matrixcontraction}
    A_N = N^{\frac{p-2}{2}} T.u^{p-2} , \qquad \text{ i.e. } \qquad (A_N)_{a,b} =  \sum_{j_1,\ldots,j_{p-2}} T_{a,b,j_1,\ldots,j_{p-2}} .
\end{equation}
We normalize 
\begin{equation}\label{eq:matrixTA}
    \Tilde{A}_N = \frac{A_N}{\sqrt{M_2 (p-2)!}} ,
\end{equation}
so that the variance of the entries is asymptotically $1/N$.
In analogy with the moment formula for the limiting spectral measure of heavy-Wigner matrices \cite{zakharevich, guionnet_abel}, we introduce the following quantities. Let $(a_r)_{r\geq 2}$ be a sequence of parameters. For every $k\geq 1$, set
\begin{equation}\label{eq:moments de Wigner lourd}
    \eta_k^{(a_r)} := \sum_{(H,o) \in \mathcal{T}_k^{\bullet}} \frac{1}{|\mathrm{Aut}(H,o)|} \sum_{P \in P_k(H)} \prod_{e \in E(H)} a_{m(P,e)},
\end{equation} 
where $\mathcal{T}_k^{\bullet}$ is the set of rooted $p$-uniform fat hypertrees with $k$ hyperedges counted with multiplicity. Moreover, $P_k(H)$ is the set of closed walks of length $k$ on the clique expansion of $H$ starting from the root such that for each hyperedge $e\in E(H)$, the walk traverses edges belonging to the hyperedge $e$ exactly $m(P,e)$ times.

\begin{theorem}\label{thm:measureHeavyWigner}
    Let $T$ be a heavy-Wigner tensor with parameters $(M_k)_{k\geq 2}$. Then the empirical spectral distribution of $\widetilde A_N$ converges weakly in probability to a probability measure $\nu$ whose moments are $(\eta_k^{(M'_r)})_{k\geq 0}$, where
\[
    M'_r:=\frac{(p-2)!^{r/2}M_r}{M_2^{r/2}},\qquad r\geq 2.
\] 

\end{theorem}
In particular, we also show that $(\eta_k^{(M'_r)})_k$ is indeed a sequence of moments of a probability measure, with $\eta_2^{(M'_r)}=1$.

Theorem \ref{thm:measureHeavyWigner} is proved in Section \ref{sec:conv_matrix}. Moreover in Section \ref{sec:limiting_law}, we prove three lemmas describing properties of the limiting law:  Lemmas \ref{lem:semicircle}, \ref{rem:no_compact_support} and \ref{lem:caract_moments}. Furthermore, we deduce from the central limit theorem on permutation invariant linear forms in the entries of a heavy-Wigner random tensor: Theorem \ref{thm: CLT for tensors}, a central limit theorem for the moments of $\Tilde{A}_N$: Theorem \ref{thm: CLT for matrix}, proved in Section \ref{sec:CLT_matrix}.

We apply Theorem \ref{thm:measureHeavyWigner} to two  random tensor models. First, Section \ref{sec:ER} is dedicated to Erd\H{o}s-Rényi hypergraphs. Section \ref{sec:ER_hypergraph} shows the local weak limit of a sparse Erd\H{o}s-Rényi hypergraph toward a Hyper Galton-Watson tree with offsprings following a Poisson law, see Theorem \ref{cvlf_er}. Section \ref{sec: limit operator contraction} is an application of Theorem \ref{thm:measureHeavyWigner} to adjacency tensors of Erd\H{o}s-Rényi hypergraphs, see Theorem \ref{thm: convergence erdos renyi}.

Finally, Section \ref{sec:levy} uses Theorem \ref{thm:measureHeavyWigner} to investigate the limiting measure of the matrix contraction of a Lévy tensor, see Theorem \ref{thm:Lévy}.


\subsection{Relation to previous works.}
Heavy-Wigner tensors extend the usual Wigner tensor model. A sequence of {\em Wigner tensors} is a sequence $(T^{(N)})_{N \geq 1}$ of symmetric tensors of order $p$ and dimension $N$ such that the entries of $N^{\frac{p-1}{2}} T^{(N)}$ are independent up to symmetry, centered and have finite moments independent of $N$ (possibly depending only on the index type).
Hence, for all $k \geq 2$, there exists $C_{k}$ such that for all $i_1<\ldots<i_p$,
$$ \E [ (N^{\frac{p-1}{2}} T_{i_1,\ldots,i_p})^{k}] = C_{k} . $$
The convention is generally to take $C_2=1$ or $C_2=1/\sqrt{(p-1)!}$.
A standard example of Wigner tensor is the Gaussian Orthogonal Tensor Ensemble (GOTE) where $T^{(N)}=X^{(N)} / N^{\frac{p-1}{2}}$ with $X^{(N)}$ a symmetric Gaussian tensor with entries distributed as Gaussian random variables with mean $0$ and variance $\sigma^2_{i_1,\ldots,i_p}=|\mathrm{Stab}(i_1,\ldots,i_p)|$.
Note that a Wigner tensor is a particular case of heavy-Wigner tensor with parameters $M_2=C_2$ and $M_k=0$ for $k\geq 3$.

\paragraph{Invariant polynomials.} Hence Theorem \ref{thm: trace injective heavy wigner} reduces, in the Wigner case, to
\[
    \lim_{N\rightarrow\infty}\E\left[\mathrm{tr}^{0}_{\mathfrak{h}^\pi}(T)\right]=\1_{\mathfrak{h}^\pi \text{ is a double hypertree}} M_2^{|E|/2}.
\]
It is the result obtained by the first author in \cite[Lemma 1]{bonnin_wigner}. Note that double hypertrees are, in their dual version, in correspondence with melonic graphs.

\paragraph{Matrix contraction.} Moreover, the matrix contractions of a Wigner tensor along a unit vector has been studied in several articles. In particular, Goulart, Couillet and Comon proved in \cite[Theorem 2]{goulart_couillet_comon} that the matrix contraction of a GOTE tensor normalized by $1/\sqrt{p!}$ converges to the semicircle law dilated by $1/\sqrt{p(p-1)}$. Au and Garza-Vargas extended the result to a general Wigner tensor in \cite{au_garza_vargas}.

For a heavy-Wigner tensor, whenever $M_k=0$ for all $k\geq 3$, the moments $\eta_k^{(M_r)}$ are equal to the Catalan numbers. In that scenario, Theorem \ref{thm:measureHeavyWigner} gives the convergence in probability of the matrix contraction of a Wigner tensor normalized by $1/\sqrt{(p-2)!}$ towards the semicircular distribution which extends the universal convergence result of \cite{au_garza_vargas}. When the law of the entries of $N^{\frac{p-1}{2}} T^{(N)}$ does not depend on $N$, the convergence in probability can be strengthened to almost sure convergence because the variance of the normalized traces of powers is $O(1/N^2)$. Moreover, if the variance of the entries is $\sigma^2_{i_1,\ldots,i_p}=|\mathrm{Stab}(i_1,\ldots,i_p)|$, the result also remains true when we contract by any sequence of unit vectors $u \in \mathbb{S}^{N-1}$ which agrees with \cite{goulart_couillet_comon}.

\paragraph{Erd\H{o}s-Rényi hypergraph.} The contraction matrices of Erd\H{o}s-Rényi hypergraphs have been studied in recent works. The convergence of the centered and normalized matrix towards the semicircular distribution has been proved, in a regime which is contained in the regime where $N^{p-1} \alpha_N \rightarrow \infty$, see the works of Mukherjee, Pal and Talukdar \cite[Theorem 1.1]{mukherjee} for hyperedges of fixed size. In this sense, Theorem \ref{thm: convergence erdos renyi} builds upon their result and characterizes a phase transition under which the convergence changes.

\section{Convergence of invariant polynomials}

In this Section, we introduce the objects under study. We then prove Proposition \ref{thm: formes lineaires invariantes}, \ref{thm: trace injective heavy wigner}, and \ref{thm: CLT for tensors}.

\subsection{Hypergraphs}\label{sec:hypergraphs}

\paragraph{Definitions.}
\begin{itemize}
    \item For $V$ a set, a {\em hyperedge} of $V$ is a multiset of elements of $V$. A directed hyperedge is a tuple of elements of $V$.

    \item A {\em simple hypergraph} is the data of $H=(V,E)$ where $V$ is a finite set of vertices, and $E$ is a finite set of hyperedgeseach consisting of distinct vertices.
    
    \item A {\em hypergraph} is the data of $H=(V,E)$, where $V$ is a finite set of vertices and $E$ is a multiset of hyperedges, without restriction on the vertices they contain. In other words, we allow multiple hyperedges and loops.
    
    \item The {\em order} (or arity) of a hyperedge is the number of vertices, counted with multiplicity, in this hyperedge, which can be written $$\vert e \vert := \sum_{v\in V}l_v(e).$$ 
    
    \item Moreover we call the {\em degree} of a vertex $v$ the number of hyperedges it belongs to, that is $$\mathrm{deg}(v):= \sum_{e\in E} l_v(e).$$
    
    \item The hypergraph is $p${\em -uniform} if all hyperedges have order $p$, that is $\vert e \vert=p$ for all $e \in E$. A $2$-uniform hypergraph is a graph. The hypergraph is $d${\em -regular} if $\mathrm{deg}(v)=d$ for all $v \in V$. 
    
    
    
    \item A hypergraph is {\em rooted} if one vertex is distinguished. 
    
    \item We define two natural graphs canonically associated to hypergraphs. First, given a simple hypergraph $H=(V,E)$ we let $G_H=(V,E_H)$ be its corresponding clique expansion. It is a simple graph with set of edges $E_H=\{ \{u,v\} : \exists e \in E, u \in e, v \in e, u \neq v \}$.
    
    \item Second, given a hypergraph $H=(V,E)$ we let $\Tilde{H}=(V_{1}\sqcup V_{2},\Tilde{E})$ denote its corresponding bipartite graph defined as follows. Vertices of type 1 are vertices of $H$ i.e. $V_{1}:=V$, vertices of type 2 are hyperedges of $H$, i.e. $V_{2}:=E$. Each couple $(v,e)\in V\times E$ with $v\in e$ induces $l_v(e)$ edges between $v\in V_1$ and $e\in V_2$. One can easily check that this is indeed a bipartite graph and that we can recover the original hypergraph from the data of the associated bipartite graph.
    
    \item A hypergraph $H$ is {\em connected} if $\Tilde{H}$ is connected.
    
    \item A hypergraph $H$ is called a {\em hypertree} if it is connected and $\Tilde{H}$ has no cycle (in the usual sense of cycle in a graph). Note that $\Tilde{H}$ has a cycle of length $2$ if and only if $H$ has a hyperedge that contains a vertex with multiplicity at least $2$. 
    
    \item A hypergraph $H$ is called a {\em hyperforest} if each of its connected components is a hypertree.
    
    \item A hypergraph $H$ is a \emph{fat hypertree} if the graph obtained from $H$ by ignoring the multiplicity of its edges is a hypertree. Furthermore, it is a {\em $k$-divisible hypertree} if it is a fat hypertree and if each edge has multiplicity divisible by $k$ and it is a {\em $k$-fold hypertree} if each edge has multiplicity exactly $k$. A $2$-fold hypertree is also called a {\em double hypertree}.

\end{itemize}

\begin{remark}
    Similarly to the case $p=2$, we may consider \emph{fully directed hypergraphs} where each hyperedge comes with an ordering of its adjacent vertices. A fully directed hypergraph is then the data of $H=(V,E)$ where $V$ is a finite set of vertices, and $E$ is a finite multiset of tuples of elements of $V$ ({\em oriented hyperedges}). Unless otherwise specified hypergraphs are undirected.
\end{remark}

\begin{lemma}[Euler's formula]\label{lem: euler}
    Let $H=(V,E)$ be a simple $p$-uniform hypergraph with $c$ connected components. Then    \begin{equation}
        \vert V \vert - (p-1) \vert E \vert - c \leq 0, 
    \end{equation}
    with equality iff $H$ is a hyperforest with $c$ connected components.
\end{lemma}

\begin{proof}
    Applying Euler's formula to the associated bipartite graph $\Tilde{H}=(\Tilde{V},\Tilde{E})$ of a given simple, $p$-uniform hypergraph $H=(V,E)$, with $c$ connected components, we obtain: 
\begin{align*}
    0 \geq \vert\Tilde{V} \vert - \vert \Tilde{E} \vert -c &= \vert V \vert + \vert E \vert - p \vert E \vert -c \\
    &= \vert V \vert - (p-1) \vert E \vert - c
\end{align*}
with equality if and only if $\Tilde{H}$ is a forest, $i.e.$ $H$ is a hyperforest.
\end{proof}

\paragraph{Adjacency tensor of a hypergraph.} 
Generalizing the $p=2$ case, where it is usual to define the adjacency matrix of a given graph, it is natural to define the adjacency tensor of a hypergraph and inquire about their relations. 

\begin{definition}
    Let $H=(V,E)$ be a $p$-uniform hypergraph with vertex set $V=\{1,\ldots,N\}$. Its adjacency tensor $T_H \in \mathrm{Sym}_p(N)$ is defined by
    $$ (T_H)_{i_1,\ldots,i_p} = \frac{1}{(p-1)!} \mathbb{1}_{ \{i_1,\ldots,i_p \} \in E }.$$
\end{definition}

Note that the choice of normalizing the tensor by $\frac{1}{(p-1)!}$ is standard and mostly aesthetic to avoid having this term in most of the results. Note that with this choice,
$$ \mathrm{deg}(v) = \sum_{i_2,\ldots,i_p} (T_H)_{v,i_2,\ldots,i_p} .$$

\begin{remark}
    We defined these notions for undirected hypergraphs in order to have symmetric tensors although it is possible to  study adjacency tensors of directed hypergraphs. If $H=(V,E)$ is a $p$-uniform fully directed hypergraph with vertex set $V=\{1,\ldots,N\}$, its adjacency tensor $T_H \in (\R^N)^{\otimes p}$ is defined by
    $$ (T_H)_{i_1,\ldots,i_p} = \mathbb{1}_{ (i_1,\ldots,i_p ) \in E }.$$
    
    Conversely, adjacency tensors of undirected hypergraphs may be seen as adjacency tensors of special directed hypergraphs: those where for each edge of the hypergraph, any permutation of the vertices of such edge is also an edge of the hypergraph.
\end{remark}

From now on, unless specified otherwise, all hypergraphs will be undirected and $p$-uniform and $p$ will always be used for the order of an edge in a hypergraph.

\paragraph{Weighted hypergraphs.}
We may now consider weighted hypergraph and extend the previous definitions to these objects. Formally, a $p$-uniform weighted hypergraph $(H,\omega)$ is the data of a $p$-uniform hypergraph $H$ and a weight function $\omega : V^p \rightarrow \C$. The adjacency tensor is then a general real tensor of order $p$.


\subsection{Trace invariants and test hypergraphs}\label{sec:trace_inv}

The notion of test-graphs comes from traffic probability (see \cite{male_traffic}) and is a well-suited framework for studying the limiting distribution of permutation invariant random matrices. Indeed, the family of linear forms associated to trace (or injective trace) of test-graphs is a complete family of permutation invariant linear forms. 
This family is thus also a generating family of $O_N$ -invariant linear forms. Indeed, if $M$ is a real symmetric matrix and $\mathfrak{g}$ is a graph, which we call \emph{test-graph} in the following, one can define a polynomial (in the entries of $M$) associated with $\mathfrak{g}=(V,E)$ as
$$ \mathrm{Tr}_{\mathfrak{g}}(M) := \sum_{\phi : V \rightarrow [N]} \prod_{e=\{v,w\} \in E} M_{\phi(v),\phi(w)},$$
which is invariant under conjugation by any permutation matrix.
It is often convenient to normalize by $N^{\# \mathfrak{g}}$ where $\# \mathfrak{g}$ is the number of connected components in the graph $\mathfrak{g}$. We write 
$$\mathrm{tr}_{\mathfrak{g}}(M)= \frac{1}{N^{\# \mathfrak{g}}} \mathrm{Tr}_{\mathfrak{g}}(M)$$ 
for the normalized trace associated to $\mathfrak{g}$.
When $\mathfrak{g}$ is the cycle of length $k$, denoted by $\mathfrak{c}_k$, which is the only connected $2$-regular graph with $k$ vertices, we get
$$ \mathrm{Tr}_{\mathfrak{c}_k}(M) = \mathrm{Tr}(M^k) = \sum_{i_1,\ldots,i_k} M_{i_1,i_2} M_{i_2,i_3} \ldots M_{i_k,i_1} .$$

These special polynomials form a basis of $O_N$-invariant linear forms. In particular, it is the case of the sum of the $k$-th powers of the eigenvalues of the matrix, which is also the $k$-th moment of the Empirical Spectral Distribution (ESD) $\mu_{M}$. If we let $\lambda_1 \geq \ldots \geq \lambda_N$ be the $N$ eigenvalues of the real symmetric matrix $M$, we have
$$ \frac{1}{N} \sum_{i=1}^N \lambda_i^k = \int \lambda^k d \mu_{M}(\lambda) = \mathrm{tr}(M^k) . $$

Numerous results arising from traffic probability can be generalized to permutation invariant random tensors. Here, we only introduce the material needed below.

A test hypergraph is a fully directed hypergraph $\hh=(V,E)$ which has all hyperedges of the same order, say $p$. It can be evaluated at a tensor $T$ of order $p$ by the following formula
\begin{equation}\label{eq: trace d'hypergraphe test}
    \mathrm{Tr}_\hh(T):=\sum_{\phi:V\rightarrow [N]}\prod_{e=(v_{1},\ldots,v_{p})\in E} T_{\phi(v_{p}),\ldots,\phi(v_{1})}.
\end{equation}
This polynomial in the entries of the tensor $T$ is called the {\em trace invariant} associated to $\hh$.

\begin{remark}
    Whenever $T$ is a symmetric tensor, we allow test hypergraphs to be undirected as the order of the vertices in a hyperedge does not change its evaluation at $T$.
\end{remark}

\begin{remark}
    In full generality, the test hypergraph comes with a label on each edge and may have hyperedges of different order and it is evaluated in a family of tensors whose order match those of their corresponding hyperedge. To simplify notation, we only mention the case where it has the same label on every hyperedge.
\end{remark}

We can decompose the sum over functions $\phi:V\rightarrow[N]$ by grouping functions $\phi$ according to which vertices take the same value, and then summing over the \emph{different} values of group of vertices. More formally, defining the injective trace of a test hypergraph as follows,
\begin{equation}\label{eq: trace injective hypergraph test}
    \mathrm{Tr}^{0}_\hh(T):=\sum_{\substack{\phi:V\rightarrow [N]\\\phi \text{ injective}}}\prod_{e=(v_{1},\ldots,v_{p})\in E} T_{\phi(v_{p}),\ldots,\phi(v_{1})},
\end{equation}
we obtain the formula
\begin{equation}\label{eq: lien trace trace injective}
    \mathrm{Tr}_\hh(T)=\sum_{\pi\in \mathcal{P}(V)}\mathrm{Tr}^{0}_{\hh^{\pi}}(T) ,
\end{equation}
where $\mathcal{P}(V)$ is the set of partitions of $V$ and $\hh^{\pi}$ is the test hypergraph obtained from $\hh$ after merging vertices in the same block of $\pi$. We show that traces (and injective traces) of test hypergraphs form a complete family of $\mathfrak{S}_N$-invariant linear forms. This is a generalization to tensors of the result in \cite{collins_male_gaudreau}.

\begin{proof}[Proof of Proposition \ref{thm: formes lineaires invariantes}]
    Again let $T$ be a symmetric tensor of order $p$ and let
\begin{align*}
    \rho_{\mathfrak{S}_N}^{k}:\mathfrak{S}_N(\mathbb{K})&\rightarrow\mathrm{End}((\mathbb{K}^{N})^{\otimes pk})\\
    U&\mapsto \left(T\mapsto U\cdot T\right)
\end{align*}
be a representation of $\mathfrak{S}_N(\mathbb{K})$ on $(\mathbb{K}^{N})^{\otimes pk}$. Assume that $R$ is a polynomial in the entries of $T$ - or equivalently a linear form on $(\mathbb{K}^{N})^{\otimes pk}$ as we have just seen - that is invariant under $\rho_{\mathfrak{S}_N}^{k}$, i.e. for all $U\in \mathfrak{S}_N(\mathbb{K})$, we have 
$$ R(T)=R(\rho_{\mathfrak{S}_N}^{k}(U)(T))=R(U\cdot T) .$$
We can again see $R$ as an element of $(\mathbb{K}^{N})^{\otimes pk}$ and write $R(T)=\langle R,T\rangle$ by the Riesz representation theorem. Therefore, we assume that $\rho_{\mathfrak{S}_N}^{k}(U)(R)=R$ for all $U\in \mathfrak{S}_N(\mathbb{K})$, or equivalently that $R=\E_\mu(U\cdot R)$, where $U$ follows some probability measure $\mu$ on $\mathfrak{S}_N(\mathbb{K})$. We denote by $\mathbb{i}\in[N]^{pk}$ a multi index of length $pk$ and $\mathrm{Ker}(\mathbb{i})$ the partition of $[pk]$ induced by $k\sim l$ if and only if $i_k=i_l$. It is easy to see that there exists $\sigma\in\mathfrak{S}_N$ such that $\sigma(i_l)=j_l$ for all $1\leq l \leq kp$ if and only if $\mathrm{Ker}(\mathbb{i})=\mathrm{Ker}(\mathbb{j})$. 

Taking $\mu$ to be the uniform distribution on $\mathfrak{S}_N$, we have that $R(\mathbb{i})=R(\mathbb{j})$ for all $\mathbb{i},\mathbb{j}$ such that $\mathrm{Ker}(\mathbb{i})=\mathrm{Ker}(\mathbb{j})$. Hence, writing for every $\pi\in \mathcal{P}(pk)$, the set of partitions of $pk$ elements, $R(\pi)$ the common value of $R(\mathbb{i})$ for every $\mathbb{i}$ such that $\mathrm{Ker}(\mathbb{i})=\pi$, we have 
$$R=\sum_{\pi\in\mathcal{P}(pk)}R(\pi)T_\pi,$$
where $T_\pi(\mathbb{i})=\delta_{\mathrm{Ker}(\mathbb{i}),\pi}$. 
Let $\HH$ be the test hypergraph with vertex set $V=\{1,\ldots,pk\}$ and edge set $E=\{e_l=(lp+1,lp+2,\ldots, (l+1)p),\,0\leq l \leq k-1\}$, and for all $\pi\in\mathcal{P}(pk)$, let $\HH^\pi$ be the test hypergraph obtained from $\HH$ by merging vertices according to $\pi$. It has $k$ hyperedges of order $p$. Since 
$$\langle T_\pi,T \rangle =\mathrm{Tr}^{0}(\HH^\pi(T)) ,$$
we conclude that injective trace of (multiple of $p$)-valent test graphs form a generating family of $\mathfrak{S}_N(\mathbb{K})$-invariant linear forms. The Möbius inversion formula proves the claim for trace of test hypergraphs as well.
\end{proof}

\begin{proof}[Proof of Corollary \ref{cor: G_N^p invariant linear forms}]
    When following the proof of Proposition \ref{thm: formes lineaires invariantes}, one can now replace measures $\mu$ on $\mathfrak S _N$ by measures $\nu$ on $G_N^p(\mathbb K)$. Taking $\nu$ to be the uniform measure on $\mathfrak S_N$ now viewed as a subgroup of $G_N^p(\K)$ leads to the same conclusion. Take now $\nu$ to be the uniform measure on $\mathbb U_N(\mathbb K)^p$ viewed as a subgroup of $G_N^p(\mathbb K)$, i.e. $\nu$ is the uniform measure on diagonal matrices having $p$-roots of unity on their diagonal. For any $\pi\in \mathcal{P}(pk)$, for all $\mathbb{i}$ such that $\mathrm{Ker}(\mathbb{i})=\pi$, we have,
    $$R(\pi)=R(\pi)\E_\nu(\prod_l U_{i_l,i_l}),$$
    where $U$ is of law $\nu$. By independence, we have
    $$\E_\nu(\prod_l U_{i_l,i_l})=\prod_{B\in\pi}\E(X^{\#B}),$$
    where $X$ is of uniform law on the $p$-th roots of unity. Hence, all blocks of $\pi$ must be of size $rp$ for some $1\leq r \leq k$, i.e. each vertex of $\hh^\pi$ has degree divisible by $p$.
\end{proof}

Let us now show Theorem \ref{prop: eigenvalues invariant by the action}. 
\begin{proof}[Proof of Theorem \ref{prop: eigenvalues invariant by the action}]
We fix $p\geq 3$ and we use notation from Definition \ref{def: G_N^p}. The proof is in two steps. First we identify an equation that is satisfied by every matrix that preserve eigenvalues of every tensor. Then we prove that $G_N^p(\K)$ is exactly the group of matrices that satisfy this relation.

Let $I$ be the identity tensor of order $p$, i.e. $I_{i_1,\ldots,i_p}=\delta_{i_1,\cdots,i_p}$, where $\delta_{i_1,\cdots,i_p}=1$ if $i_1=\cdots=i_p$ and 0 otherwise. Let $U$ be a matrix such that $U \cdot I = I$. This gives for $1 \leq i_1,\ldots,i_p \leq N$,
    \[
    \sum_{i=1}^N \prod_{k=1}^p U_{i,i_k} = \delta_{i_1,\ldots,i_p} .
    \]
Let $T$ be a symmetric tensor of order $p$ and let $x \in \K^N \setminus \{0 \}$ such that $T.x^{p-1}=\lambda x^{\circ p-1}$. Let $y = U^{-1} x$, we have for $1\leq s \leq N$,
    \begin{align*}
        [(U \cdot T).y^{p-1}]_s &= \sum_{i_1,\ldots,i_{p-1}} (U \cdot T)_{i_1,\ldots,i_{p-1},s} \prod_{k=1}^{p-1} y_{i_k} \\
        &= \sum_{\genfrac{}{}{0pt}{}{i_1,\ldots,i_{p-1}}{j_1,\ldots,j_p}}  T_{j_1,\ldots,j_{p}} (\prod_{k=1}^{p-1} U_{j_k,i_k}) U_{j_p,s} \prod_{k=1}^{p-1} (\sum_{t=1}^N U^{-1}_{i_k,t} x_t) \\
        &= \sum_{\genfrac{}{}{0pt}{}{i_1,\ldots,i_{p-1}}{j_1,\ldots,j_p}}  T_{j_1,\ldots,j_{p}} \sum_{t_1,\ldots,t_{p-1}} (\prod_{k=1}^{p-1} U_{j_k,i_k} U^{-1}_{i_k,t_k} x_{t_k}) U_{j_p,s} \\
        &= \sum_{j_1,\ldots,j_p}  T_{j_1,\ldots,j_{p}} \sum_{t_1,\ldots,t_{p-1}} (\prod_{k=1}^{p-1} \delta_{j_k,t_k} x_{t_k}) U_{j_p,s} \\
        &= \sum_{j_1,\ldots,j_p}  T_{j_1,\ldots,j_{p}} \prod_{k=1}^{p-1}  x_{j_k} U_{j_p,s} ,
    \end{align*}
Using the eigenvalue relation, we get
    \begin{align*}
        [(U \cdot T).y^{p-1}]_s &= \lambda \sum_{j_p} x_{j_p}^{p-1} U_{j_p,s} \\
        &= \lambda \sum_{j_p} (\sum_{k=1}^N U_{j_p,k} y_k)^{p-1} U_{j_p,s} \\
        &= \lambda \sum_{r_1,\ldots,r_{p-1}} y_{r_1} \ldots y_{r_{p-1}} \sum_{j_p} U_{j_p,r_1} \ldots U_{j_p,r_{p-1}} U_{j_p,s}, 
    \end{align*}
    where 
    $$\sum_{j_p} U_{j_p,r_1} \ldots U_{j_p,r_{p-1}} U_{j_p,s} = (U\cdot I)_{r_1,\ldots,r_{p-1},s} =  I_{r_1,\ldots,r_{p-1},s} = \delta_{r_1,\ldots ,r_{p-1},s} .$$
    Hence, we finally have
    \[ 
    [(U \cdot T).y^{p-1}]_s = \lambda y_s^{p-1}, 
    \]
    so $\lambda$ is an eigenvalue of $U \cdot T$. The converse is true by applying $U^{-1}$.

    We now show that $G_N^p(\mathbb K)$ is indeed the group of matrices that satisfy the relation $U\cdot I=I$. Let $U$ be a matrix such that $U\cdot I=I$, note that it implies $U^{-1}=(U^{\circ(p-1)})^{T}$. We denote by $v^{(i)}$ the $i$-th column of $U$. The equation $U\cdot I=I$ can be written as
    \begin{equation}\label{eq: u}
        \sum_j v^{(i_1)}_j\cdots v^{(i_p)}_j=\delta_{i_1,\cdots,i_p}.
    \end{equation}
    For any $1\leq i \leq N$, let $b_i$ be the bilinear form on $\mathbb{K}^{N}$ defined by $b_i(x,y)=\sum_{j}(v_j^{(i)})^{p-2}x_j y_j$. First, note that by Equation \eqref{eq: u}, we have $b_i(v^{(k)},v^{(l)})=\delta_{k,l,i}$ since $p\geq 3$. Since $U$ is invertible, the family $\mathbf{v}=(v^{(k)})_{k=1,\cdots,N}$ is a basis of $\mathbb{K}^{N}$ and $\mathrm{Mat}_{\mathbf{v}}(b_i)=E_{i,i}$. On the other hand, if $\mathbf{e}=(e_k)_{k=1,\cdots, N}$ is the canonical basis of $\mathbb{K}^{N}$, we have $\mathrm{Mat}_{\mathbf{e}}(b_i)=\mathrm{Diag}((v_k^{(i)})^{p-2},\,k=1,\cdots,N).$ Comparing the ranks of the two matrices, since $p\geq 3$, we have that the number of non-zero entries in each vector $v^{(i)}$ is exactly one. Furthermore, it is easy to see that this non-zero entry is a $p$-th root of the unity by looking at $b_i(v^{(i)},v^{(i)})$. Hence, $U$ is the product of a permutation matrix with a diagonal matrix whose entries are in $\mathbb{U}_{p}(\mathbb{K})$. The equation $V_\sigma D V_\pi D'= V_\sigma V_\pi (V_\pi^{-1}D V_\pi)D'=V_{\sigma\circ\pi} (V_\pi^{-1}D V_\pi)D'$ exhibits the structure of semi-direct product.
    
    Conversely, if $U$ and $U\cdot T$ have same eigenvalues for every symmetric $T\in(\mathbb K^n)^{\otimes p}$, then $U=V_\sigma D$ for some permutation $\sigma\in\mathfrak S_N$ and some diagonal matrix of $p$-roots of unity. Indeed take $T=x^{\otimes p}$ for any $x\in\C^N$ and note that $U\cdot x^{\otimes p}=(U^T x)^{\otimes p}$. Furthermore, since $U$ preserves eigenvalues, Theorem 1 of \cite{QI2005} guarantees that it also preserves the trace:
    \begin{equation}\label{eq: trace préservée}
        \sum_{i=1}^{N}x_i^p=\sum_{i=1}^{N}((U^T x)_i)^p.
    \end{equation}
    Differentiating once with respect to each $x_j$ yields $x^{\circ(p-1)}=U(U^T x)^{\circ(p-1)}$. Taking $x=e_j$ shows that $U$ is invertible since every element of the basis is in the image of $U$. Then, differentiating with respect to two different coordinates $i\neq j$ gives
    \begin{equation}
        \sum_{k=1}^{N}U_{i,k}U_{j,k}((U^T x)_k)^{p-2}=0.
    \end{equation}
    For any $1\leq k \leq N$, choosing $x=(U^T)^{-1}e_k$, we have $U_{i,k}U_{j,k}=0$ which shows that $U$ has at most one non zero entry per column. Thus, invertibility forces $U=V_\sigma D$ for some diagonal matrix $D$ with non-zero diagonal coefficients. Finally, taking $x=e_{\sigma(k)}$ in Equation \eqref{eq: trace préservée} for all $1\leq k \leq N$, we get that $D^p=I$, which concludes the proof.
    \end{proof}

\subsection{Convergence of the trace invariants}\label{sec:conv_trace_inv}

We establish a convergence result for the injective trace of test hypergraphs evaluated at heavy-Wigner tensors.

\begin{proof}[Proof of Theorem \ref{thm: trace injective heavy wigner}]
Let $\mathfrak{h}$ be a connected test hypergraph, 
$$\E\left[\mathrm{tr}^{0}(\mathfrak{h}(T))\right]=\E\left[\frac{1}{N}\sum_{\substack{\phi:V\rightarrow [N]\\\phi \text{ injective}}}\prod_{e=(v_{1},\ldots,v_{p})\in E}T_{\phi(v_{p}),\ldots,\phi(v_{1})}\right].$$
Let $\overline{E}$ be the set of hyperedges without their multiplicity. To be more precise, we define a map $f:E\subset\N\times V^{p}\rightarrow\overline{E},\;e=(j,v_{1},\ldots,v_{p})\mapsto(v_{1},\ldots,v_{p})$. For each element $a$ of $\overline{E}$, we denote by $r(a)$ the cardinality of $f^{-1}(a)$. With these notation, we obtain 
$$\E\left[\mathrm{tr}^{0}(\mathfrak{h}(T))\right]=N^{-1}\sum_{\substack{\phi:V\rightarrow [N]\\\phi \text{ injective}}}\prod_{a=(v_{1},\ldots,v_{p})\in \overline{E}}\E\left[(T_{\phi(v_{p}),\ldots,\phi(v_{1})})^{r(a)}\right].$$
Using the assumption on the moments of the entries of a heavy-Wigner tensor, we obtain
\begin{align*}
    \E\left[\mathrm{tr}^{0}(\mathfrak{h}(T))\right] &= N^{-1}\sum_{\substack{\phi:V\rightarrow [N]\\\phi \text{ injective}}}\prod_{a\in \overline{E}}\left(M_{r(a)}+o(1)\right)N^{1-p}\\
    &=N^{-1}\sum_{\substack{\phi:V\rightarrow [N]\\\phi \text{ injective}}} N^{|\overline{E}|(1-p)}\left(\prod_{a\in \overline{E}}M_{r(a)}+o(1)\right).
\end{align*}
Note now that the number of such injective maps $\phi:V\rightarrow[N]$ is $\frac{N!}{(N-|V|)!}=N^{|V|}\left(1+O\left(\frac{1}{N}\right)\right)$. Since $c(\mathfrak{h}):=\prod_{a\in \overline{E}}M_{r(a)}$ does not depend on $\phi$ anymore, we obtain 
$$\E\left[\mathrm{tr}^{0}(\mathfrak{h}(T))\right] = N^{|V|-(p-1)|\overline{E}|-1}(c(\mathfrak{h})+o(1)).$$
Now using Euler's identity we have $|V|-(p-1)|\overline{E}|-1\leq 0$ with equality if and only if $\overline{\mathfrak{h}}$ is a hypertree, i.e. $\mathfrak{h}$ is a fat hypertree.
Furthermore, the constant term $c(\mathfrak{h})$ corresponds exactly to the constant term in equation \eqref{eq:limite_trace_inj_heavy_wig}.
\end{proof}

\begin{corollary}\label{cor:trace_injective_heavy_wig}
    Let $\mathfrak{h}=(V,E)$ be a connected $p$-uniform $2$-regular test hypergraph, $\pi$ be a partition of its vertices and $T$ be a heavy-Wigner tensor. Then, we have
    \begin{equation}
        \lim_{N\rightarrow\infty}\E\left[\mathrm{tr}^{0}_{\mathfrak{h}^\pi}(T)\right]=\1_{\mathfrak{h}^\pi \text{ is a } 2\text{-divisible hypertree}} \prod_{r \in \N} M_{r}^{q_{r}}.
    \end{equation}
\end{corollary}

\begin{proof}[Proof of Corollary \ref{cor:trace_injective_heavy_wig}]
Let $\mathfrak{h}$ be a connected $p$-uniform $2$-regular hypergraph and let $\pi$ be a partition of its vertices, such that $\E\left[\mathrm{tr}^{0}(\mathfrak{h}^\pi(T))\right]$ is not vanishing. Then by Theorem \ref{thm: trace injective heavy wigner} $\HH^\pi$ is a fat hypertree. Note that a fat hypertree obtained as a quotient of a $2$-regular hypergraph necessarily has hyperedges of even multiplicity. Indeed it is the case for the hyperedges containing the leaves, and then for any hyperedge of the hypertree by an immediate induction.
\end{proof}


 \subsection{Central limit theorem on the trace invariants}\label{sec:CLT_tensor}

In this section we establish a central limit theorem for injective traces of test hypergraphs evaluated at a heavy-Wigner tensor. 

\begin{theorem}\label{thm: CLT for tensors}
    Let $$(Z_N(\hh))_\hh:=\left(\frac{1}{\sqrt{N}}(\mathrm{Tr}^{0}_\hh(T)-\E\mathrm{Tr}^0_\hh(T))\right)_\hh,$$ be a process indexed by $p$-uniform test hypergraphs $\hh$. This process converges to a centered Gaussian process $(z_\hh(T))_\hh$ whose covariance is given by
    $$\E\left(z_{\hh_1}(T)z_{\hh_2}(T)\right)=\sum_{\hh\in\mathcal{P}_\#(\hh_1,\hh_2)}\lim_{N\rightarrow\infty}\E\frac{1}{N}\mathrm{Tr}^{0}_{\hh}(T),$$
    where $\mathcal{P}_\#(\hh_1,\hh_2)$ is the set of hypergraphs obtained by gluing vertices of $\hh_1$ and $\hh_2$ together so that they have at least one hyperedge in common.
\end{theorem}
 
The method of this proof is inspired by \cite{benaych_georges_male_guionnet}, where a similar convergence for heavy-Wigner matrices is established.
\begin{proof}[Proof of Theorem \ref{thm: CLT for tensors}]
    The Gaussian distribution being characterized by its moments, it suffices to prove convergence of the joint moments. Let $\mathfrak{h}_{1}=(V_{1},E_{1}),\ldots,\mathfrak{h}_{n}=(V_{n},E_{n})$ be connected test hypergraphs. First, we write
    \begin{align*}
        \E\left[\prod_{i=1}^{n}Z_{N}(\mathfrak{h}_{i})\right]&=N^{-n/2}\sum_{\substack{\phi_{1},\ldots,\phi_{n}\\\phi_{j}:V_{j}\rightarrow[N]\text{ injective}}}\underbrace{\E\left[\prod_{i=1}^{n}\left(\prod_{e\in E_{i}}T_{\phi_{j}(e)}-\E\left(\prod_{e\in E_{i}}T_{\phi_{j}(e)}\right)\right) \right]}_{\omega_{N}(\phi_{1},\ldots,\phi_{n})},\\
        &=\sum_{\pi\in\mathcal{P}(n)}\sum_{\substack{\sigma\in\mathcal{P}(V_{1},\ldots,V_{n})\\\overline{\sigma}=\pi}}N^{-n/2}\sum_{(\phi_{1},\ldots,\phi_{n})\in S_{\sigma}}\omega_{N}(\phi_{1},\ldots,\phi_{n}),
    \end{align*}
where we used the following notation:
    \begin{itemize}
        \item $\mathcal{P}(V_{1},\ldots,V_{n})$ is the set of partitions $\sigma$ of the disjoint union of the $V_{i}$'s for $i\in [n]$ such that each block of $\sigma$ contains at most one vertex of each $V_{i}$.
        \item For a given $\sigma \in\mathcal{P}(V_{1},\ldots,V_{n})$, $\overline{\sigma}$ is the element of $\mathcal{P}(n)$ defined by $i\sim_{\overline{\sigma}}j\Leftrightarrow\exists v\in V_{i},v'\in V_{j}$ such that $v\sim_{\sigma}v'$. In that way, $\overline{\sigma}$ tells us which of the $V_{i}$'s are connected through $\sigma$.
        \item $S_{\sigma}$ is the set of injective maps $(\phi_{i}:V_{i}\rightarrow[N])_{i\in [n]}$ such that for any $v\in V_{i},v'\in V_{j}$, we have $\phi_{i}(v)=\phi_{j}(v')\Leftrightarrow v\sim_{\sigma}v'$. Equivalently, $S_{\sigma}$ is the set of injective functions $\Phi:\left(\bigsqcup_{i=1}^{n}V_{i}\right)^{\sigma}\rightarrow [N]$.
    \end{itemize}
    Since $T$ is invariant in law by permutation, the quantity $\omega_{N}(\phi_{1},\ldots,\phi_{n})$ does not depend on $\phi_{1},\ldots,\phi_{n}$ but only on $\sigma$. For $(\phi_{1},\ldots,\phi_{n})\in S_{\sigma}$, we define $\omega_{N}(\phi_{1},\ldots,\phi_{n})=:\omega_{N}(\sigma)$.
    
    Let us introduce $\mathcal{P}_{\#}(V_{1},\ldots,V_{n})$ the subset of $\mathcal{P}(V_{1},\ldots,V_{n})$ containing all the partitions $\sigma$ such that for all $i\in[n]$ there exists an edge $e=(v_{1},\ldots,v_{p})$ of $E_{i}$ which is identified with another edge $e'=(w_{1},\ldots,w_{p})$ of $E_{j}$ (with $j\neq i$), i.e. $v_{1}$ is in the same block as $w_{1}$, $v_{2}$ is in the same block as $w_{2}$ and so on. By independence of the entries of $T$ and by the centering of the terms in $\omega_{N}(\sigma)$, if $\sigma\in\mathcal{P}(V_{1},\ldots,V_{n})\setminus\mathcal{P}_{\#}(V_{1},\ldots,V_{n})$, then $\omega_{N}(\sigma)=0$. Furthermore, since $|S_{\sigma}|=\frac{N!}{\left(N-|\left(\bigsqcup_{i=1}^{n}V_{i}\right)^{\sigma}|\right)!}=N^{|\left(\bigsqcup_{i=1}^{n}V_{i}\right)^{\sigma}|}(1+O(1/N))$, we obtain
    \begin{equation}\label{eq: etape 1}
        \E\left[\prod_{i=1}^{n}Z_{N}(\mathfrak{h}_{i})\right]=\sum_{\pi\in\mathcal{P}(n)}\sum_{\substack{\sigma\in\mathcal{P}_{\#}(V_{1},\ldots,V_{n})\\\overline{\sigma}=\pi}}N^{-n/2+|\left(\bigsqcup_{i=1}^{n}V_{i}\right)^{\sigma}|}\omega_{N}(\sigma)(1+O(1/N)).
    \end{equation}
    Let us now focus on the term $\omega_{N}(\sigma)$. Let $\sigma\in\mathcal{P}_{\#}(V_{1},\ldots,V_{n})$, let $(\phi_{1},\ldots,\phi_{n})\in S_{\sigma}$ and let $\Phi:\left(\bigsqcup_{i=1}^{n}V_{i}\right)^{\sigma}\rightarrow [N]$ denote the associated injective map. Expanding $\omega_{N}(\sigma)$, we get
    \begin{equation}\label{eq: omegaN v1}
        \omega_{N}(\sigma)=\sum_{B\subset\{1,\ldots,n\}}(-1)^{n-|B|}\E\left[\prod_{j\in B}\prod_{e\in E_{j}}T_{\phi_{j}(e)} \right]\prod_{j\notin B}\E\left[\prod_{e\in E_{j}}T_{\phi_{j}(e)}\right].
    \end{equation}
    For $B\subset\{1,\dots,n\}$, we denote by $\mathfrak{h}_{B,\sigma}$ the test hypergraph with set of vertices $V_{B,\sigma}:=\left(\bigsqcup_{j\in B}V_{j}\right)^{\sigma}$ and set of edges $E_{B,\sigma}$ obtained from $\bigsqcup_{j\in B}E_{j}$ induced by $\sigma$. Denoting finally $\overline{E}_{B,\sigma}$ the set of hyperedges of $ \mathfrak{h}_{B}$ when forgetting their orientation and their multiplicity, we have
    \begin{align*}
        \E\left[\prod_{j\in B}\prod_{e\in E_{j}}T_{\phi_{j}(e)}\right]&=N^{-(p-1)|\overline{E}_{B,\sigma}|}\prod_{\bar e \in \overline{E}_{B,\sigma}}\E\left[N^{p-1}T_{\Phi(\bar e)}^{m(\bar e)}\right],\\
        &=N^{-(p-1)|\overline{E}_{B,\sigma}|}\prod_{\bar e \in \overline{E}_{B,\sigma}}M_{m(\bar e)}(1+o(1)).
    \end{align*}
    Plugging this into Equation \eqref{eq: omegaN v1}, we obtain
    \begin{equation}\label{eq: omegaN v2}
         \omega_{N}(\sigma)=\sum_{B\subset\{1,\ldots,n\}}(-1)^{n-|B|}N^{-(p-1)\left(|\overline{E}_{B,\sigma}|+\sum_{j\notin B}|\overline{E}_{\{j\},\sigma}|\right)}\prod_{\bar e \in \overline{E}_{B,\sigma}}M_{m(\bar e)}\prod_{j\notin B}\prod_{\bar e \in \overline{E}_{\{j\},\sigma}}M_{m(\bar e)}(1+o(1)),
    \end{equation}
    Putting together Equations \eqref{eq: omegaN v2} and \eqref{eq: etape 1} we get
    \begin{equation*}
        \E\left[\prod_{i=1}^{n}Z_{N}(\mathfrak{h}_{i})\right]=\sum_{\pi\in\mathcal{P}(n)}\sum_{\substack{\sigma\in\mathcal{P}_{\#}(V_{1},\ldots,V_{n})\\\overline{\sigma}=\pi}}\sum_{B\subset\{1,\ldots,n\}}(-1)^{n-|B|}N^{\eta}\delta^{0}(\sigma,B)(1+o(1)),
    \end{equation*}
    where 
    $$\delta^{0}(\sigma,B)=\prod_{\bar e \in \overline{E}_{B,\sigma}}M_{m(\bar e)}\prod_{j\notin B}\prod_{\bar e \in \overline{E}_{\{j\},\sigma}}M_{m(\bar e)},$$
    and $\eta = \eta_{1}+\eta_{2}+\eta_3$ are the following quantities:
    \begin{equation*}
        \eta_{1}=(p-1)\left(|\overline{E}_{\sigma}|-|\overline{E}_{B,\sigma}|-\sum_{j\notin B}|\overline{E}_{j}|\right)\leq 0
    \end{equation*}
    with equality if and only if $B=\{1,\ldots,n\}$;
    \begin{equation*}
        \eta_{2}=-c(\sigma)+\left|\left(\bigsqcup_{i=1}^{n}V_{i}\right)^{\sigma}\right|-(p-1)|\overline{E}_{\sigma}|\leq0,
    \end{equation*}
    using Euler's inequality where $c(\sigma)$ is the number of connected components of $(\sqcup\mathfrak{h}_i)^{\sigma}$ and equality holds if and only if each connected component is a fat hypertree; 
    \begin{equation*}
        \eta_3=c(\sigma)-n/2\leq 0,
    \end{equation*}
    because $\sigma\in\mathcal{P}_{\#}(V_{1},\ldots,V_{n})$ implies $\overline{\sigma}\in\mathcal{P}_{\geq2}(n)$ (the set of partitions with all blocks of size at least 2) with equality if and only if $\overline{\sigma}\in\mathcal{P}_{2}(n)$: the set of partition whose blocks are all of size 2. 

    Hence, the only non vanishing contribution is when the hypergraphs are paired and that the resulting connected components are all fat hypertrees. Hence, $\delta^0$ factorizes as 
    \begin{equation*}
        \delta^{0}(\sigma,\{1,\ldots,n\})=\prod_{\mathfrak{h}\in\mathcal{CC}((\sqcup\mathfrak{h}_{i})^{\sigma})}\lim_{N\rightarrow\infty}\E\frac{1}{N}\mathrm{Tr}^{0}_\mathfrak{h}(T),
    \end{equation*}
    where $\mathcal{CC}((\sqcup\mathfrak{h}_{i})^{\sigma})$ is the set of connected components of the quotient by $\sigma$ of the disjoint union of the original hypergraph. All in all, we have the much simpler expression
    \begin{align*}
        \E\left[\prod_{i=1}^{n}Z_{N}(\mathfrak{h}_{i})\right]&=\sum_{\pi\in\mathcal{P}_{2}(n)}\sum_{\substack{\sigma\in\mathcal{P}_{\#}(V_{1},\ldots,V_{n})\\\overline{\sigma}=\pi}}\delta^{0}(\sigma,\{1,\ldots,n\})(1+o(1))\\
        &=\sum_{\pi\in\mathcal{P}_{2}(n)}\prod_{\{i,j\}\in\pi}\sum_{\sigma\in\mathcal{P}_{\#}(V_{i},V_{j})}\lim_{N\rightarrow\infty}\E\frac{1}{N}\mathrm{Tr}^{0}_{(\mathfrak{h}_i\sqcup\mathfrak{h}_j)^{\sigma}}(T)(1+o(1)),
    \end{align*}
    which can be expressed as follows. Let 
    $$M^{(2)}(\hh_1,\hh_2):=\sum_{\sigma\in\mathcal{P}_{\#}(V_{1},V_{2})}\lim_{N\rightarrow\infty}\E\frac{1}{N}\mathrm{Tr}^{0}_{(\mathfrak{h}_1\sqcup\mathfrak{h}_2)^{\sigma}}(T),$$ be the limiting covariance. Then for $n\geq 3$, we have
    \begin{equation*}
        \E\left[\prod_{i=1}^{n}Z_{N}(\mathfrak{h}_{i})\right]=\sum_{\pi\in\mathcal{P}_{2}(n)}\prod_{\{i,j\}\in\pi}M^{(2)}(\hh_i,\hh_j)+o(1),
    \end{equation*}
    which characterizes Gaussian distributions.
\end{proof}


\section{Contracted heavy-Wigner tensor}
In this Section, we prove Theorem \ref{thm:measureHeavyWigner} and its applications. We first prove the convergence of the moments and then bound the variance. Then, we apply the result to contractions of Erd\H{o}s-Rényi hypergraphs and Lévy tensors.

\subsection{Convergence of the moments}\label{sec:conv_matrix}

\begin{proposition}\label{prop:moments}
   
    Let $T$ be a heavy-Wigner tensor with parameters $(M_k)_{k\geq 2}$ and let $\Tilde{A}$ be defined by Equation \eqref{eq:matrixTA}. Then for any $k\geq 1$, 
    $$ \lim_{N\rightarrow\infty} \E \left[\frac{1}{N}\mathrm{Tr}(\Tilde{A}^k)\right] = \eta_k^{(M'_k)} , $$
    where 
    $$(M'_k)_{k\geq2} := \left(\frac{(p-2)!^{k/2} M_k}{M_2^{k/2}}\right)_{k\geq 2} .$$
\end{proposition}

\begin{proof}
For $k\geq 1$, consider the $k$-th moment of the matrix $A$ given by
\begin{align*}
    \E \left[\frac{1}{N}\mathrm{Tr}(A^k)\right] 
    &= \frac{1}{N} \sum_{i_1,\ldots,i_k} \E A_{i_1,i_2} A_{i_2,i_3} \ldots A_{i_{k-1},i_k}A_{i_1,i_k} \\
    &=\frac{1}{N}\sum_{i_1,\ldots i_k}\sum_{j^1_1,\ldots ,j^1_{p-2}}\cdots\sum_{ j^k_1,\ldots,j^k_{p-2}}\E T_{i_1,i_2,j^1_1,\ldots,j^1_{p-2}}\cdots T_{i_k,i_1,j^k_1,\ldots,j^k_{p-2}}
\end{align*}
In order to compute the expectation on the last line, we define the following (fully directed) test hypergraph $\hh_k$ from the directed cycle of length $k$ by transforming each edge into a hyperedge by adding $p-2$ new vertices for each edge. More formally, denoting $\cc_k=(V,E)$, the set of vertices of $\hh_k$ is
$$ \mathbf{V} =V \sqcup W := V \sqcup \bigsqcup_{e\in E}\{e^{1},\ldots,e^{p-2}\}$$
whereas the set of hyperedges is
$$ \mathbf{E}:=\bigsqcup_{e=(w,v)\in E}\{(e^{p-2},\ldots,e^{1},w,v)\} .$$
Now, we can rewrite the $k$-th moment of $A$ by summing over the different possible partitions of $\hh_k$ as follows
\begin{equation}\label{eq:traceA}
    \begin{aligned}
    \E \left[\frac{1}{N}\mathrm{Tr}(A^k)\right] 
    &= \sum_{\pi\in\mathcal{P}(\mathbf{V})}\frac{1}{N}\E \mathrm{Tr}^{0}_{\hh_k^{\pi}}(T),\\
    &=\sum_{\pi\in\mathcal{P}(\mathbf{V})}\1_{\hh_k^{\pi}\text{ is a fat hypertree}}\prod_{\bar e\in\overline{\mathbf{E}^{\pi}}}M_{m(\bar e)}(1+o(1)),
\end{aligned}
\end{equation}
where we used Equation \eqref{eq:limite_trace_inj_heavy_wig} for the last line. Note that since $\hh_k$ comes from the cycle of length $k$, each $\hh_k^{\pi}$ is naturally endowed with a closed walk of length $k$ on its clique expansion graph which we denote by $\gamma(\hh_k^\pi)$ from now on. In other words, $\gamma(\hh_k^\pi)$ is the image of $\cc_k^\pi$ in $\hh_k^{\pi}$. Let us write $\gamma=(v_0,v_1,\ldots,v_k)$ with $v_0=v_k=o$. We denote by $e_1,\cdots, e_k$ the hyperedges of $\hh_k$ in the cyclic order so that $v_{t-1}$ and $v_t$ are the images under $\pi$ of the two last vertices of $e_t$. Hence for a fat hyperedge $\bar e$ of $\hh_k^{\pi}$, we have 
$$m(\bar e)=\#\{1\leq t \leq k, e_t\overset{\pi}{\mapsto} \bar e\},$$
where $e_t\overset{\pi}{\mapsto}e$ means that the (unordered version of the) image of the hyperedge $e_t$ in $\overline{\hh_k^\pi}$ is $\bar e$.
We sum over the isomorphism classes of rooted fat hypertrees, which can be seen as simple hypertrees with a label on each hyperedge that correspond to the multiplicity. Furthermore, each hyperedge must have multiplicity at least 2 and the sum of the multiplicities must equal $k$. Let $\mathcal{T}_k^\bullet$ be the set of these isomorphism classes. For $(H,o)\in\mathcal{T}_k^\bullet$, we write $P_k(H)$ the set of closed walks of length $k$ on its clique expansion graph $G_H$ starting from the root, that traverses exactly $m(\bar e)$ times an edge from the hyperedge $e$. We thus have
$$\lim_{N\rightarrow\infty}\E \left[\frac{1}{N}\mathrm{Tr}(A^k)\right]=\sum_{(H,o)\in\mathcal{T}_k^\bullet}\sum_{\gamma\in P_k(H)}\prod_{\bar e \in \overline{E(H)}}M_{m(\bar e)}\#\{\pi\in\mathcal{P}(\mathbf{V}),\,(\hh_k^\pi,\gamma(\hh_k^\pi))\equiv((H,o),\gamma)\},$$
where $(\hh_k^\pi,\gamma(\hh_k^\pi))\equiv((H,o),\gamma)$ means that there is a rooted isomorphism between $\hh_k^\pi$ and $(H,o)$ that sends $\gamma(\hh_k^\pi)$ to $\gamma$. We now compute this combinatorial factor.

We claim that, for fixed $(H,o)$ and fixed $\gamma\in\mathcal P_k(H,o)$,
\[
\#\left\{(\pi,\varphi):\pi\in\mathcal P(\mathbf V),\,\varphi:\hh_k^\pi\to H\text{ is a rooted isomorphism and}\,\varphi(\gamma(\hh_k^\pi))=\gamma\right\}=(p-2)!^k.
\]
Indeed, suppose first that such a pair $(\pi,\varphi)$ is given. For
$1\leq t\leq k$, the hyperedge $\mathbf e_t$ is mapped to the hyperedge $e_t(\gamma)$ of $H$. The two cycle vertices $x_{t-1}$ and $x_t$ are
mapped to $v_{t-1}$ and $v_t$. Hence the remaining vertices
$e_t^1,\ldots,e_t^{p-2}$ must be mapped bijectively onto $e_t(\gamma)\setminus\{v_{t-1},v_t\}$. Thus $(\pi,\varphi)$ determines a bijection $\theta_t:\{1,\ldots,p-2\}\longrightarrow e_t(\gamma)\setminus\{v_{t-1},v_t\}$ for every $t=1,\ldots,k$.

Conversely, given such a family of bijections $(\theta_t)_{1\leq t\leq k}$, it is easy to reconstruct the pair $(\pi,\varphi)$. Indeed, the bijection is equivalent to an ordering of the first $p-2$ vertices of $e_t$, hence the partition $\pi$ and the isomorphism $\varphi$ follow easily.

Now, for a fixed partition $\pi$ such that $\hh_k^\pi$ is isomorphic to $(H,o)$ as a rooted hypergraph, the number of rooted isomorphisms
$\varphi:\hh_k^\pi\to H$ is exactly $|\mathrm{Aut}(H,o)|$. Consequently,
$$
\lim_{N\to\infty} \E\left[\frac1N\operatorname{Tr}(A^k)\right]=\sum_{[H,o]\in\mathcal T_k^\bullet}\frac{(p-2)!^k}{|\operatorname{Aut}(H,o)|}\sum_{\gamma\in\mathcal P_k(H,o)}\prod_{e\in E(H)}M_{m_e(\gamma)}.$$
Finally, since
$$\widetilde A = \frac{A}{\sqrt{M_2(p-2)!}},$$ 
we obtain the wanted result.

\end{proof}

\begin{remark}\label{rem:eta2}
    We always have
    $$\eta_2^{(M'_r)}=1.$$
    Indeed, $M'_2=(p-2)!$. We now evaluate $\eta_2^{(M'_r)}$ from the hypertree moment formula. For total multiplicity $2$, the only contributing rooted fat hypertree consists of one $p$-uniform hyperedge $e$ of multiplicity $2$, rooted at one of its vertices. Its rooted automorphism group has cardinality $(p-1)!$. Its clique expansion is the complete graph on the $p$ vertices of $e$, and the closed walks of length $2$ from the root are precisely
\[
    o\to v\to o,
    \qquad v\in e\setminus\{o\}.
\]
There are $p-1$ such walks. For each of them, the unique hyperedge is used twice, so the contribution is $M'_2(B)=(p-2)!$. Hence
\[
\eta_2^{((M'_r))} = \frac{p-1}{(p-1)!}(p-2)! =1.
\]
\end{remark}


\subsection{Properties of the limiting law}\label{sec:limiting_law}
\begin{lemma}\label{lem:semicircle}
    If $M_2=1$ and $M_{2r}=0$ for $r > 1$, then $\nu$ is the semicircle law supported on $[-2,2]$. 
\end{lemma}

\begin{proof}[Proof of Lemma \ref{lem:semicircle}]
    We assume that $M_2=1$ and $M_{2r}=0$ for every $r > 1$  (and therefore $M_k=0$ for all $k\geq 3$ by the Cauchy-Schwarz inequality).
    We start from Equation \eqref{eq:traceA},
    $$ \E \left[\frac{1}{N}\mathrm{Tr}(A^k)\right] 
    =\sum_{\pi\in\mathcal{P}(\mathbf{V})}\1_{\hh_k^{\pi}\text{ is a fat hypertree}}\prod_{\bar e\in\overline{\mathbf{E}^{\pi}}}M_{m(\bar e)}(1+o(1)). $$
    We deduce that $\E \left[\frac{1}{N}\mathrm{Tr}(A^{2k+1})\right]=0$ and
    $$ \E \left[\frac{1}{N}\mathrm{Tr}(A^{2k})\right] 
    =\sum_{\pi\in\mathcal{P}(\mathbf{V})}\1_{\hh_k^{\pi}\text{ is a double hypertree}}M_2^{k}(1+o(1)). $$
    Considering $ \mathbf{V} =V \sqcup W := V \sqcup \bigsqcup_{e\in E}\{e^{1},\ldots,e^{p-2}\}$ and $ \mathbf{E}:=\bigsqcup_{e=(w,v)\in E}\{(e^{p-2},\ldots,e^{1},w,v)\}$, we see that $\hh_k^{\pi}$ may be a double hypertree only if $V^\pi$ is a rooted plane double tree. For the vertices of $\mathbf{V}\setminus V$, $\pi$ must only match the $p-2$ other vertices of two hyperedges already sharing two vertices from $V$. There are $(p-2)!$ choices for each pair of hyperedges.
    Therefore, since the number of rooted plane double trees is given by the Catalan numbers and since $ \Tilde{A} = A / \sqrt{M_2(p-2)!} $,
    we have 
    $$\E \left[\frac{1}{N}\mathrm{Tr}(\Tilde{A}^{2k})\right] 
    =\frac{1}{k+1} \binom{2k}{k} (p-2)!^k \frac{M_2^k}{[M_2(p-2)!]^k} . $$
    Finally, we obtain
    $$ \eta_k^{(M'_r)} = \left\{
    \begin{array}{ll}
        \frac{1}{m+1} \binom{2m}{m} & \mbox{if } k=2m \\
        0 & \mbox{otherwise.}
    \end{array}
    \right.$$
    These are the moments of the semicircle law. on $[-2,2]$
\end{proof}

\begin{lemma}\label{rem:no_compact_support}
    If $M_{2r} > 0$ for some $r >1$, then $\nu$ is not compactly supported.
\end{lemma}

\begin{proof}[Proof of Lemma \ref{rem:no_compact_support}]
    Assume $M_{2k}>0$ for some $k>1$. For $s\geq 1$, consider the rooted star hypertree with $s$ hyperedges incident to the root, each used with multiplicity $2r$. A closed walk of length $2ks$ can be obtained by choosing an ordered word of $rs$ excursions, with each of the $s$ hyperedges selected exactly $k$ times; each excursion goes from the root to a fixed vertex in the selected hyperedge and back. This gives at least
\[
    \frac{(ks)!}{(k!)^s}
\]
walks before quotienting by rooted automorphisms. Since the automorphism group has size at most $s!((p-1)!)^s$, the moment satisfies
\[
    \eta_{2ks}^{(M_r)}
    \geq
    \frac{(ks)!}{s!\,(k!)^s((p-1)!)^s}M_{2k}^s.
\]
By Stirling's formula, $(\eta_{2ks}^{(M_r)})^{1/(2ks)}$ is unbounded as $s\to\infty$. A compactly supported probability measure has moments bounded by $R^k$ for some $R$, so the limiting law cannot have compact support.
\end{proof}

\begin{lemma}\label{lem:caract_moments}
    If $M_r = O(a^r)$ for some $a>0$, then the probability measure with moment sequence $(\eta_k^{(M_r)})_{k\geq 0}$ is uniquely characterized by its moments.
\end{lemma}
\begin{proof}
    Note that for each rooted hypertree with at most $k$ hyperedges, there are at most $(pk)^k$ walks of length $k$, since, at each step, the number of possible choices for the next vertex is bounded by the total number of vertices. Moreover, there are 
    $$ \frac{1}{pk+1} \binom{pk}{k} \leq (pk)^k $$
    rooted hypertrees (up to isomorphism). We obtain
    \begin{align*}
        \sum_{k\geq 1} \left[\eta_{2k}^{(M_r)}\right]^{-1/2k} \geq \frac{1}{2p^2a} \sum_{k\geq 1} \frac{1}{k} =\infty .
    \end{align*}
    Then, by Carleman's criterion, we get the desired result.
\end{proof}


\subsection{Variance of the moments}\label{sec:variance_matrix}

\begin{proposition}\label{prop:variance}
    Let $T$ be a heavy-Wigner tensor and $\Tilde{A}$ as defined in Equation \eqref{eq:matrixTA}, then for any $k\geq 1$, 
    $$\mathrm{Var}\left(\frac{1}{N}\mathrm{Tr}(\Tilde{A}^k)\right) = O \left(\frac{1}{N}\right) . $$
\end{proposition}
\begin{remark}
    This is a crude estimate of the variance of the moments of $\Tilde{A}$. Theorem \ref{thm: CLT for matrix} gives a precise convergence result for the process obtained by centering and rescaling the moments of $\Tilde{A}$.
\end{remark}
We can now prove Theorem \ref{thm:measureHeavyWigner}.
\begin{proof}[Proof of Theorem \ref{thm:measureHeavyWigner}]
    Proposition \ref{prop:moments} guarantees that the convergence holds in expectation. 
    Theorem \ref{thm:measureHeavyWigner} follows from Proposition \ref{prop:variance} and Markov's inequality.
\end{proof}
\begin{proof}[Proof of Proposition \ref{prop:variance}]
    Let $T$ be a heavy-Wigner tensor with parameters $(M_k)_{k\geq 2}$. Denote by
    $$Z^k_N :=\mathrm{Var}\left(\frac{1}{N}\mathrm{Tr}(\Tilde{A}^k)\right)=\E\left[\left(\frac{1}{N}\mathrm{Tr}(\Tilde{A}^k)-\E\frac{1}{N}\mathrm{Tr}(\Tilde{A}^k)\right)^{2}\right] $$ 
    and recall that $\hh_k$ is the hypergraph consisting of $k$ hyperedeges forming a cycle (i.e. obtained from a cycle by adding $p-2$ new vertices for each edge and joining them with the edge to form a $p$ hyperedge). By direct computation, we have
    \begin{equation}
        \mathrm{Tr}(\Tilde{A}^k)=\mathrm{Tr}_{\hh_k}(T)\frac{1}{((p-2)!M_2)^{k/2}}.
    \end{equation}
    Hence, we can rewrite the variance as 
    \begin{align*}
        \mathrm{Var}\left(\frac{1}{N}\mathrm{Tr}(\Tilde{A}^k)\right)&=\frac{1}{N((p-2)!M_2)^k}\E\left(\left(\frac{1}{\sqrt{N}}(\mathrm{Tr}_{\hh_k}(T)-\E\mathrm{Tr}_{\hh_k}(T))\right)^{2}\right).
    \end{align*}
    Note now that the expectation converges to the covariance of some Gaussian process by Theorem \ref{thm: CLT for tensors}, hence is bounded.
    \end{proof}


\subsection{CLT for the contracted tensor}\label{sec:CLT_matrix}

We can deduce from Theorem \ref{thm: CLT for tensors} a central limit theorem for the moment of the spectral measure of the matrix $\Tilde{A}$.

\begin{theorem}\label{thm: CLT for matrix}
    Let $T$ be a heavy-Wigner tensor with parameters $(M_k)_{k\geq 2}$ and $\Tilde{A}$ defined in Equation \eqref{eq:matrixTA}. The process
    $$(Z_N(k))_{k\geq 1}:=\left(\frac{1}{\sqrt{N}}(\mathrm{Tr}(\Tilde{A}^k)-\E\mathrm{Tr}(\Tilde{A}^k))\right)_{k\geq 1},$$
    converges to a centered Gaussian process $(z(k))_{k\geq 1}$ with covariance given by
    $$\E(z(k_1)z(k_2))=\frac{1}{((p-2)!M_2)^{(k_1+k_2)/2}}\sum_{\pi_1,\pi_2}\E\left(z_{\hh_{k_1}^{\pi_1}}(T)z_{\hh_{k_2}^{\pi_2}}(T)\right),$$
    where $z_{\hh}(T)$ is the process of Theorem \ref{thm: CLT for tensors} and the sum over $\pi_1$ (resp. $\pi_2$) ranges over partitions of the vertices of $\hh_{k_1}$ (resp. $\hh_{k_2}$) which is the hypergraph with $k_1$ (resp. $k_2$) hyperedges forming a cycle.
\end{theorem}

Note that Proposition \ref{prop:variance} is also an immediate corollary of Theorem \ref{thm: CLT for matrix}.

\begin{proof}[Proof of Theorem \ref{thm: CLT for matrix}]
    Let $1\leq k_1,\ldots,k_n$ be integers and recall that $\hh_k$ is the hypergraph consisting of $k$ hyperedeges that forms a cycle (i.e. obtained from a cycle by adding $p-2$ new vertices for each edge and joining them with the edge to form a $p$ hyperedge). By direct computation, we have
    \begin{equation}
        \mathrm{Tr}(\Tilde{A}^k)=\mathrm{Tr}_{\hh_k}(T)\frac{1}{((p-2)!M_2)^{k/2}}.
    \end{equation}
    Hence, denoting $K=\sum k_i$, we have
    \begin{align*}
        \E(\prod_{i=1}^n Z_N(k_i))&=\frac{N^{-n/2}}{((p-2)!M_2)^{K/2}}\E\left[\prod_{i=1}^{n}(\mathrm{Tr}_{\hh_{k_i}}(T)-\E\mathrm{Tr}_{\hh_{k_i}}(T))\right]\\
        &=\sum_{\pi_1,\ldots,\pi_n}\frac{1}{((p-2)!M_2)^{K/2}}\E\left(\prod_{i=1}^{n}Z_N(\hh_{k_i}^{\pi_i})\right)
    \end{align*}
    where, $Z_N$ in the last line is defined in Theorem \ref{thm: CLT for tensors}. Recalling that $M^{(2)}$ is the limiting covariance for injective traces of the heavy-Wigner tensor $T$, we obtain the characterization of a Gaussian distribution:
    $$\E(Z_N(k_1)\ldots Z_N(k_n))=\sum_{\pi\in\mathcal{P}_2(n)}\prod_{\{i,j\}\in\pi}\sum_{\substack{\pi_1\in\mathcal{P}(V(\hh_{k_1})),\\\pi_2\in\mathcal{P}(V(\hh_{k_2}))}}\frac{M^{(2)}(\hh_{k_i}^{\pi_i},\hh_{k_j}^{\pi_j})}{((p-2)!M_2)^{(k_1+k_2)/2}}+o(1).$$
\end{proof}





\section{Application 1: Erd\H{o}s-Rényi hypergraph}\label{sec:ER}
In this Section, we first define an Erd\H{o}s-Rényi hypergraph and study its local weak limit in the sparse regime. We then study the spectral statistics of the matrix contraction of the adjacency tensor of a Erd\H{o}s-Rényi hypergraph, see the works of Bordenave, Lelarge and Salez \cite{bordenave_salez}, of Erdös, Yau, Knowles and Yin \cite{erdos_yau_knowles_yin} or the survey of Guionnet \cite{guionnet_survey}.

\subsection{Sparse Erd\H{o}s-Rényi hypergraphs and local weak limit}\label{sec:ER_hypergraph}

\begin{definition}[Erd\H{o}s-Rényi hypergraph]
    A $p$-Erd\H{o}s-Rényi hypergraph on $N$ vertices with parameter $\alpha \in (0,1)$ is constructed as follows. Let the set of vertices be $V=\{1,\ldots,N\}$ and for each $1\leq i_1<\ldots<i_p\leq N$, the edge $\{i_1,\dots,i_p\}$ is an edge of $H$ with probability $\alpha$ independently for all such $p$-tuple. 
\end{definition}
In the following, we will consider a sequence of Erd\H{o}s-Rényi hypergraphs with parameters $(\alpha_N)$ and $T=(T_N)$ the sequence of their adjacency tensors, but we often supress the dependence on $N$ in our notation. We also define $c_N=N^{p-1} \alpha_N$. Note that the expectation of the entries is 
$$\E [T_{i_1,\ldots,i_p}] = \frac{1}{(p-1)!} \alpha_N $$
if $i_1,\ldots,i_p$ are distinct (and $0$ otherwise). We denote by $\E[T]$ the tensor having entries $\E [T_{i_1,\ldots,i_p}]$. The variance of the entries with distinct indices is 
$$ \mathrm{Var}[T_{1,\ldots,p}] = \alpha_N (1-\alpha_N) /(p-1)!^2  . $$

\begin{lemma}\label{lem:ER_heavy}
    Let $T$ be the adjacency tensor of an Erd\H{o}s-Rényi hypergraph with parameter $\alpha_N =c_N /N^{p-1}$ such that $c_N$ converges towards $c>0$. Then, the tensor
    $$ \Tilde{T} := \frac{T-\E[T]}{\sqrt{N^{p-1} \mathrm{Var}[T_{1,\ldots,p}]} } $$
    is a heavy-Wigner tensor with asymptotic moments $M_k= \left(\frac{1}{c}\right)^{\frac{k}{2}-1}$ for all $k\geq 2$.
\end{lemma}
\begin{proof}
Let $1\leq i_1 < \ldots < i_p \leq N$. The entries are independent up to symmetries by construction. 
It is also clear that $\E[\Tilde{T}_{i_1,\ldots,i_p}]=0$. Now let $k\geq 2$. Since the random variable $T_{i_1,\ldots,i_p}$ takes only the values $0$ or $1$, it is also the case of $T_{i_1,\ldots,i_p}^k$, so we have
\begin{align*}
    N^{p-1} \E[\Tilde{T}_{i_1,\ldots,i_p}^k] &= N^{p-1} \frac{(p-1)!^{k}}{(p-1)!^k} \times \frac{(-\alpha_N)^k (1 - \alpha_N) + (1-\alpha_N)^k \alpha_N}{N^{(p-1)k/2} \alpha_N^{k/2} (1-\alpha_N)^{k/2}}  \\
    &=\frac{c_N}{c_N^{k/2}} (1+o(1)) .
\end{align*}
We get the result.
\end{proof}

\begin{lemma}\label{lem:ER}
    Let $T$ be the adjacency tensor of an Erd\H{o}s-Rényi hypergraph with parameter $\alpha_N =c_N /N^{p-1}$ such that $c_N \rightarrow \infty$. Then, the tensor 
    $$ \Tilde{T} := \frac{T-\E[T]}{\sqrt{N^{p-1} \mathrm{Var}[T_{1,\ldots,p}]} } $$
    is a heavy-Wigner tensor with $M_2=1$ and $M_k=0$ for $k \geq 3$.
\end{lemma}
\begin{proof}
    The proof follows the same argument as that of Lemma \ref{lem:ER_heavy} using the fact that if $c_N \rightarrow \infty$, then $c_N^{-\frac{k-2}{2}} \rightarrow 0$ as soon as $k \geq 3$.
\end{proof}

We conclude with a brief discussion of local weak convergence for hypergraphs. We use notation and results from \cite{delgosha_anantharam}.

\begin{theorem}\label{cvlf_er}
    Let $H_N$ be a sequence of Erd\H{o}s-Rényi hypergraphs with parameter $\alpha_N=c/N^{p-1}$. Then $H_N$ converges locally weakly towards the uniform Galton-Watson hypertree with offspring distribution $\mathrm{Poi}\left(c / (p-1)!\right)$ in the sense of \cite{delgosha_anantharam}.
\end{theorem}

\begin{proof}[Proof of Theorem \ref{cvlf_er}]
Let $H_N$ be a sequence of Erd\H{o}s-Rényi hypergraphs with parameter $\alpha_N:=c/N^{p-1}$ on $N$ vertices. More formally, let $\mathcal{H}_{N,p}$ be the set of simple $p$-uniform hypergraphs on the set of vertices $\{1,\cdots,N\}$ and consider $\mu_N\in\mathcal{P}(\mathcal{H}^{*})$, i.e. a probability measure on the set of equivalence classes of simple connected rooted hypergraphs defined as
$$\mu_N=\frac{1}{N}\sum_{i=1}^{N}\sum_{H\in \mathcal{H}_{N,p}}\alpha_N^{|E(H)|}(1-\alpha_N)^{\binom{N}{p}-|E(H)|}\delta_{[H(i),i]},$$ where $[H(i),i]$ is the equivalence class of the connected component of $i$ in $H$ rooted at $i$ (for the equivalence relation of isomorphism between rooted unlabeled hypergraphs).

Let $\mu$ be the law of a Uniform Galton-Watson Hypertree ($UGWH$) with offspring distribution
$$\gamma = \mathrm{Poi\left(\frac{c}{(p-1)!}\right)},$$ 
so that each vertex has a random number of offspring hyperedges distributed as $\gamma$. Note that $\hat{\gamma},$ the size-biased offspring distribution that each vertex except the root follows coincides with $\gamma$. Here, we only consider $p$-hyperedges so $\gamma$ is therefore a probability distribution on $\N$.
We want to prove that for any $r\geq 1$ and any rooted hypergraph $(X,v)$ of depth at most $r$, 
$$\mu_N(A_r(X,v))\rightarrow\mu(A_r(X,v)),\quad \text{where } A_r(X,v):=\{[H,o]\in\mathcal{H}^{*},\,(H,o)_r\equiv (X,v)\},$$ where $\equiv$ means that the two rooted hypergraphs are isomorphic.
Let $(X,v)$ be a $p$-uniform hypergraph of depth at most $r$. Then 
\begin{align*}
    \mu_N(A_r(X,v))&=\frac{1}{N}\sum_{i=1}^{N}\sum_{H\in \mathcal{H}_{N,p}}\alpha_N^{|E(H)|}(1-\alpha_N)^{\binom{N}{p}-|E(H)|}\1_{(H(i),i)_r\equiv(X,v)},\\
    &=\frac{1}{N}\sum_{i=1}^{N}\sum_{\substack{\phi:V(X)\rightarrow[N],\\\text{injective},\\\phi(v)=i}}\sum_{H\in\mathcal{H}_{N,p}}\frac{\1_{(H(i),i)_r=(\phi(X),i)}}{|\mathrm{Aut}(X,v)|}\alpha_N^{|E(H)|}(1-\alpha_N)^{\binom{N}{p}-|E(H)|},
\end{align*}
where $\phi(X)$ is the image of the hypergraph $X$ under $\phi$ and the equality between the two hypergraphs means that they are equal, not only isomorphic. The factor $1/|\mathrm{Aut}(X,v)|$ compensate the fact that given the image of $\phi$, there are $|\mathrm{Aut}(X,v)|$ ways of labeling the graph with indices from the image of $\phi$ that give exactly the same graph.
Note that $(H(i),i)_r=(\phi(X),i)\implies \prod_{e\in E(X)}\1_{\phi(e)\in E(H)}=1$,
so that 
\begin{align*}
    |\mu_N(A_r(X,v))| \leq \sum_{H\in\mathcal{H}_{N,p}}\alpha_N^{|E(H)|}(1-\alpha_N)^{\binom{N}{p}-|E(H)|} \frac{1}{N} \mathrm{Tr}^0_{X} (T_{H}),
\end{align*}
where $T_H$ is the unnormalized, uncentered adjacency tensor of the hypergraph $H$.
Let $T$ be the (random) adjacency tensor of the Erd\H{o}s-Rényi hypergraph $H_N \sim H(N, c/N^{p-1})$, that we \emph{do not normalize} and \emph{do not center}. 
Then, we have
\begin{align*}
    |\mu_N(A_r(X,v))| \leq \E \frac{1}{N} \mathrm{Tr}^0_{X} (T)  .
\end{align*}
Following the same proof as that of Theorem \ref{thm: trace injective heavy wigner}, we have for any \emph{simple} hypergraph $\hh$, 
$$\lim_{N\rightarrow\infty}\E\frac{1}{N}\mathrm{Tr}^0_\hh(T)=0,$$
as soon as $\hh$ contains a cycle.
If we restrict ourselves to hypergraphs of fixed maximal depth, then $\mu_N$ is asymptotically supported on hypertrees.
Since $\mu$ is also supported on hypertrees, as a consequence of Lemma $4$ in \cite{delgosha_anantharam}, it is enough to show that $\mu_N(A_r(X,v))\rightarrow\mu(A_r(X,v))$ for any $r\geq 1$ and any rooted hypertree $(X,v)$ of depth at most $r$.

Therefore, we consider $(X,v)$ a $p$-uniform hypertree of depth at most $r$, and we recall
\begin{align*}
    \mu_N(A_r(X,v))=\frac{1}{N}\sum_{i=1}^{N}\sum_{\substack{\phi:V(X)\rightarrow[N],\\\text{injective},\\\phi(v)=i}}\sum_{H\in\mathcal{H}_{N,p}}\frac{\1_{(H(i),i)_r=(\phi(X),i)}}{|\mathrm{Aut}(X,v)|}\alpha_N^{|E(H)|}(1-\alpha_N)^{\binom{N}{p}-|E(H)|}.
\end{align*}
Since $\phi$ is injective, $|\phi(V(X))| =|V(X)|$. Let $V_\partial\subset V(\phi(X))$ denote the set of vertices at distance exactly $r$ from the root in $\phi(X)$ and note that summing over hypergraphs $H$ that satisfy the condition $(H(i),i)_r=(\phi(X),i)$ amounts to summing over hypergraphs $H$ that can be obtained from $(\phi(X),i)$ by adding hyperedeges $e=\{v_1,\cdots,v_p\}$ with $\{v_1,\cdots,v_p\}\cap \phi(V(X))\subset V_\partial$. Note that there are 
$$L:=\binom{N-|V(X)|+|V_\partial|}{p}$$
such hyperedges. This naturally decomposes $E(H)$ into $E_1\sqcup E_2$ where $E_1= \phi(E(X))$ and $E_2$ is a set of edges as just described. Hence, we have
\begin{align*}
    \mu_N(A_r(X,v))&=\frac{1}{N}\sum_{i=1}^{N}\sum_{\substack{\phi:V(X)\rightarrow[N],\\\text{injective},\\\phi(v)=i}}\sum_{E_2}\frac{\alpha_N^{|E(X)|}(1-\alpha_N)^{\binom{N}{p}-|E(X)|}}{|\mathrm{Aut}(X,v)|}\alpha_N^{|E_2|}(1-\alpha_N)^{-|E_2|},\\
    &=\frac{1}{N}\sum_{i=1}^{N}\sum_{\substack{\phi:V(X)\rightarrow[N],\\\text{injective},\\\phi(v)=i}}\frac{\alpha_N^{|E(X)|}(1-\alpha_N)^{\binom{N}{p}-|E(X)|}}{|\mathrm{Aut}(X,v)|}\sum_{j=0}^{L} \binom{L}{j}\left( \frac{\alpha_N}{1-\alpha_N}\right)^j,\\
    &=\frac{1}{N}\sum_{i=1}^{N}\sum_{\substack{\phi:V(X)\rightarrow[N],\\\text{injective},\\\phi(v)=i}}\frac{\alpha_N^{|E(X)|}(1-\alpha_N)^{\binom{N}{p}-|E(X)|-L}}{|\mathrm{Aut}(X,v)|},\\
    &=\frac{1}{|\mathrm{Aut}(X,v)|}N^{-1+|V(X)|-(p-1)|E(X)|}(1-\alpha_N)^{\binom{N}{p}-|E(X)|-L}c^{|E(X)|}(1+O(1/N)),
\end{align*}
where the sum over $E_2$ ranges over sets of hyperedges as described above. Since $(X,v)$ is fixed, $|V(X)|,|V_\partial|$ and $|E(X)|$ are independent of $N$ and 
$$\binom{N}{p}-|E(X)|-\binom{N-|V(X)|+|V_\partial|}{p}=\frac{1}{(p-1)!}(|V(X)|-|V_\partial|)N^{p-1}(1+o(1)),$$ 
so that we have
$$ \mu_N(A_r(X,v))=\frac{1}{|\mathrm{Aut}(X,v)|}e^{-\frac{c}{(p-1)!}(|V(X)|-|V_\partial|)} c^{|E(X)|}(1+o(1)).$$

Let us now compute $\mu(A_r(X,v))$. For each vertex $w$ at depth strictly less than $r$ in $(X,v)$, denote by $d_w$ the number of offspring hyperedges of $w$. In the UGWH, we have
$$ \P (N_w = d_w) = e^{-\frac{c}{(p-1)!}} \frac{c^{d_w}}{(p-1)!^{d_w}d_w!} .$$
When we forget the order of the offspring hyperedges, and of the $(p-1)$ vertices in each of these hyperedges, then we have
\begin{align*}
    \mu(A_r(X,v)) &= \frac{1}{|\mathrm{Aut}(X,v)|}\prod_{w \in V(X) \setminus V_\partial} \left( e^{-\frac{c}{(p-1)!}} \frac{c^{d_w}}{(p-1)!^{d_w}d_w!} \right) (p-1)!^{d_w} d_w!  \\
    &= \frac{1}{|\mathrm{Aut}(X,v)|}e^{-\frac{c}{(p-1)!}(|V(X)|-|V_\partial |)} c^{\sum_{w \in V(X) \setminus V_\partial} d_w},
\end{align*} 
where division by $|\mathrm{Aut}(X,v)|$ corrects for overcounting. Since $\sum_{w \in V(X) \setminus V_\partial} d_w  = |E(X)|$, we get
$$\mu(A_r(X,v)) =\frac{1}{|\mathrm{Aut}(X,v)|} e^{-\frac{c}{(p-1)!}(|V(X)|-|V_\partial |)} c^{|E(X)|}.$$ 
Finally, we have $\mu_N(A_r(X,v))\rightarrow\mu(A_r(X,v))$, so we conclude $\mu_N \Rightarrow \mu .$

\end{proof}

\subsection{Limiting operator for the contracted tensor}\label{sec: limit operator contraction}

We again consider $u =\mathbf{1}/ \sqrt{N} \in \mathbb{S}^{N-1}$ and we define 
\begin{equation}\label{eq:ERmatrixTA}
    \Tilde{A}_N = \frac{N^{\frac{p-2}{2}} \Tilde{T}.u^{p-2}}{\sqrt{(p-2)!}} 
\end{equation} 

For $c>0$, let $\mu_c$ be the measure with moments given for $k\geq 1$ by
$$ \eta_{k} = \sum_{H \in \mathcal{T}_k^{\bullet}} \frac{1}{|\mathrm{Aut(H)}|} \sum_{\gamma \in P_k(H)} \prod_{e \in E} (p-2)!^{m_\gamma(e)} \left(\frac{1}{c} \right)^{\frac{m_\gamma(e)-2}{2}} , $$
where $\mathcal{T}_k^{\bullet}$, $P_k(H)$ and $m_\gamma(e)$ have been defined in Equation \eqref{eq:moments de Wigner lourd}. Note that $\mu_c$ has unbounded support and is uniquely characterized by its moments.

We construct a random operator $\mathcal{A}$ from the random hypergraph which is known as the (renormalized) clique expansion of $\mathcal{T}$. Let $\mathcal{T}$ be a UGWH with offspring distribution $\mathrm{Poi}(c/(p-1)!)$ and for each realization of $\mathcal{T}$, let $\mathcal{A}$ be the operator defined by 
\begin{equation}\label{eq: operateur limit matrix contraction ER}
    (\mathcal{A}\psi)(v)=\sqrt{\frac{(p-2)!}{c}}\sum_{e\ni v}\sum_{\substack{w\in e,\\w\neq v}}\psi(w),
\end{equation}
where $\psi\in L^2(V)$ and $v\in V$. This defines a random operator and we show in the next theorem that its expected spectral measure is $\mu_c$.

\begin{theorem}\label{thm: convergence erdos renyi}
    Let $H_N$ be a $p$-uniform Erd\H{o}s-Rényi hypergraph with parameter $\alpha_N =c_N /N^{p-1}$ and let $\mu_N$ be the ESD of $\Tilde{A}_N$ as defined previously. Then, as $N \rightarrow \infty$, 
    \begin{itemize}
        \item[-] if $c_N \rightarrow \infty$, then $\mu_N \rightarrow \mu_{\mathrm{sc}}$ weakly in probability, where $\mu_{\mathrm{sc}}$ is the semicircle law supported on $[-2,2]$,
        \item[-] if $c_N \rightarrow c >0$, then $\mu_N \rightarrow \mu_{c}$ weakly in probability, where $\mu_c$ is a measure uniquely characterized by its moments, depending only on $c$ and with unbounded support. 
        \item[-] Denote by $\mu_{\mathcal{A}}$ the rooted spectral measure of $\mathcal{A}$ defined in \eqref{eq: operateur limit matrix contraction ER}. Then $\mu_c=\mathbb E\mu_\mathcal A$.
    \end{itemize}
    
\end{theorem}

\begin{proof}[Proof of Theorem \ref{thm: convergence erdos renyi}]
    The first point is a consequence of Lemma \ref{lem:ER} and Theorem \ref{thm:measureHeavyWigner}. Note that the law of the entries depends on $N$ so we have a variance bound of order $1/N$ and therefore the weak convergence only holds in probability.
    The first part of the second point is a consequence of Lemma \ref{lem:ER_heavy} and Theorem \ref{thm:measureHeavyWigner}.

Let us now show the third point and recall that
$$\int t^k\dd\E\mu_\mathcal{A}=\E\langle \delta_o,\mathcal{A}^k \delta_o\rangle,$$
where $o$ is the root of the hypergraph $\mathcal T$.
Since $\mu_c$ is determined by its moments, it is enough to show that the expected number of closed walks of length $k$ from the root on $\mathcal{A}$ (correctly normalized) corresponds to $\eta_k$. Recall that 
\begin{align*}
    \eta_k&=\sum_{(H,o)\in\mathcal{T}^{\bullet}_k}\frac{(p-2)!^{|E(H)|}}{|\mathrm{Aut}(H,o)|}\sum_{\gamma\in P_k(G_H)}\prod_{e\in \overline{E(H)}}\left(\frac{(p-2)!}{c}\right)^{m_\gamma(e)/2-1},\\
    &=\left(\frac{(p-2)!}{c}\right)^{k/2}\sum_{(H,o)\in\mathcal{T}^{\bullet}_k}\frac{c^{|E(H)|}}{|\mathrm{Aut}(H,o)|}|P_k(H)|. 
\end{align*}
Set $\mathcal{A}'=\sqrt{c/(p-2)!}\mathcal{A}$, we show that 
$$\E\langle \delta_o,(\mathcal{A}')^k\delta_o\rangle = \sum_{(H,o)\in\mathcal{T}^{\bullet}_k}\frac{c^{|E(H)|}}{|\mathrm{Aut}(H,o)|}|P_k(H)|.$$
We compute the expected number of closed walks on $\mathcal{A}'$ depending on their supporting fat hypertree.

Fix some $(H,o)\in \mathcal{T}_k^\bullet$ and a path $\gamma\in P_k(H)$. In order to embed it into a UGWH, one must choose at each vertex $v$, $r_v$ distinct hyperedges among $D_v$ where $r_v$ is the number of offspring hyperedges of $v$ in $H$ and $D_v$ is a Poisson random variable. The number of ordered choices is $D_v(D_v-1)\cdots (D_v-r_v+1)$ and $\E D_v(D_v-1)\cdots (D_v-r_v+1)=(c/(p-1)!)^{r_v}.$ Furthermore, for each of these hyperedges, their $p-1$ vertices (other than $v$) may also be injected in $(p-1)!$ different ways. Hence, the expected number of labeled rooted injections from $(H,o)$ into $\mathcal{A}'$ is $c^{\sum_{v}r_v}=c^{|E(H)|}$. Since each unlabeled hypergraph is counted $|\mathrm{Aut}(H,o)|$ times in this way, we obtain the desired formula.

\end{proof}

\section{Application 2: Lévy tensor}\label{sec:levy}

The spectral theory of heavy-tailed random matrices goes back to the work of Bouchaud and Cizeau \cite{bouchaud_cizeau}. They introduced Lévy matrices as random matrix models with infinite-variance entries. Such models exhibit limiting spectral distributions that differ fundamentally from the semicircle law. These distribution were studied and characterized by Ben Arous and Guionnet \cite{ben_arous_guionnet} and Bordenave, Caputo and Chafaï \cite{bordenave_caputo_chafai}. We work in this framework and prove Theorem \ref{thm:Lévy} in this section.

Let $(x_{i_{1},\ldots,i_{p}},1\leq i_{1}\leq i_{2}\leq\ldots\leq i_{p}<\infty)$ be an infinite $p$-array of i.i.d. complex variables. Denote by $X_{N}$ the symmetric tensor of order $p$ given by 
\begin{equation}
    X_{N}(i_{1},\ldots,i_{p})=x_{i_{\sigma(1)},\ldots,i_{\sigma(p)}},
\end{equation}
where $\sigma$ is a permutation such that $i_{\sigma(1)}\leq\ldots\leq i_{\sigma(p)}$.

Let $\alpha\in(0,2)$. We assume that the common distribution of the $|x_{\mathbf{i}}|$'s belongs to the domain of attraction of an $\alpha$-stable law, i.e. there exists a slowly varying function $L$ such that 
\begin{equation}
    \P(|x_{\mathbf{i}}|\geq u)=\frac{L(u)}{u^{\alpha}}.
\end{equation}
We introduce the normalizing constant 
\begin{equation}
    a_{N}:=\mathrm{inf}\left(u,\;\P(|x_{\mathbf{i}}|\geq u)\leq\frac{1}{N}\right).
\end{equation}
One can check that there is a slowly varying function $L_{0}(N)$ such that 
\begin{equation}
    a_{N^{p-1}}=L_{0}(N)N^{\frac{p-1}{\alpha}}.
\end{equation}
A Lévy tensor $T_{N}$ is defined as $T_{N}:=a_{N^{p-1}}^{-1}X_{N}$. Recall that we work under the additional tail-balance condition
$$\theta:=\lim_{t\rightarrow\infty}\frac{\P(x_{\mathbf{i}}>t)}{\P(|x_{\mathbf{i}}|>t)} \in [0,1].$$
For $B >0$, we let $X_N^B$ denote the truncated tensors with entries $x_{i_{1},\ldots,i_{p}} \mathbf{1}_{| x_{i_{1},\ldots,i_{p}} | \leq B a_{N^{p-1}}}$ and define its centered normalized version by
$$T_N^B := a_{N^{p-1}}^{-1}X_{N}^B - \mathbb{E} \left[ a_{N^{p-1}}^{-1}X_{N}^B \right] .$$
We also define $u=\mathbf{1}/\sqrt{N} \in \mathbb{S}^{N-1}$ and 
$$ \Tilde{A}^B_N := \frac{N^{\frac{p-2}{2}} T_N^B .u^{p-2} }{\sqrt{M_2 (p-2)!}} .$$

\begin{theorem}\label{thm:Lévy}
    With the notation above, the following statements hold:
    \begin{enumerate}
        \item For each fixed $B>0$, $\mu_{\Tilde{A}_N^B}$ converges weakly in probability, as $N \rightarrow \infty$, to a probability measure $\mu^B$ uniquely determined by its moments and having unbounded support.
        \item The measures $\mu^B$ converge weakly to a limiting measure  $\mu_{\alpha,\theta}$ as $B \rightarrow \infty$.
    \end{enumerate}
    In general the odd limiting moments depend on $\theta$; in the symmetric-tail case $\theta=1/2$, all odd moments vanish and the limiting law is symmetric.
\end{theorem}

The proof of Theorem \ref{thm:Lévy} relies on Theorem \ref{thm: trace injective heavy wigner} and the following two crucial lemmas.

\begin{lemma}\label{lem:levyheavy}
    The tensor $a_{N^{p-1}}^{-1}X_{N}^B$ is a heavy-Wigner tensor with parameters
    $$ M_{k} = \frac{\alpha}{k-\alpha} B^{k-\alpha} (\theta + (-1)^k (1-\theta)) $$
    In particular,
    $ M_{2k} = \frac{\alpha}{2k-\alpha} B^{2k-\alpha}$ and $ M_{2k+1} = (2 \theta -1) \frac{\alpha}{2k+1-\alpha} B^{2k+1-\alpha}$. 
\end{lemma}
\begin{proof}
    For $s\geq 2$, Karamata's theorem (see Theorem VIII.9.2 in \cite{feller}) yields the following estimate for the moments of the truncated entries:  
    \begin{align*}
        \E [|x_{i_{1},\ldots,i_{p}} |^s \1_{|x_{i_{1},\ldots,i_{p}} |\leq B a_{N^{p-1}}} ] 
        & \sim \frac{\alpha}{s-\alpha} (B a_{N^{p-1}})^s \mathbb{P}(|x_{i_{1},\ldots,i_{p}} |> B a_{N^{p-1}}) \\
        & \sim \frac{\alpha}{s-\alpha} B^{s-\alpha} a^s_{N^{p-1}} \mathbb{P}(|x_{i_{1},\ldots,i_{p}} |>a_{N^{p-1}}) , 
    \end{align*} 
    where the second equivalence simply follows from $$\mathbb{P}(|x_{i_{1},\ldots,i_{p}} |>Ba_{N^{p-1}})= \frac{L(Ba_{N^{p-1}})}{B^\alpha a_{N^{p-1}}^\alpha} \sim \frac{L(a_{N^{p-1}})}{B^\alpha a_{N^{p-1}}^\alpha} = \frac{\mathbb{P}(|x_{i_{1},\ldots,i_{p}} |>a_{N^{p-1}})}{B^\alpha} . $$
    Since $N^{p-1}\mathbb{P}(|x_{i_{1},\ldots,i_{p}} |>a_{N^{p-1}}) \to 1$, the even moments have the stated limit. The odd moments follow from the tail-balance assumption. Subtracting the mean does not change the limits for $s\geq 2$, because the centered mean contributes only lower-order terms to $N^{p-1}\mathbb{E}[(T_N^B)^s]$.
\end{proof}

\begin{corollary}\label{cor:convLévy}
    As $N \rightarrow \infty$, we have weak convergence
    $ \mathbb{E} \left[\mu_{\Tilde{A}^B} \right] \rightarrow \mu^B $,
    where $\mu^B$ has unbounded support and is uniquely determined by its moments which are given by the $\eta_k^{(M'_r)}$ for $k\geq 1$, where
    $$ M'_r= \frac{(p-2)!^{r/2} (2- \alpha)}{r-\alpha} \left( \frac{2-\alpha}{\alpha} B^{\alpha}\right)^{r/2-1} (\theta + (-1)^r (1-\theta)). $$
\end{corollary}
\begin{proof}
    We apply Theorem \ref{thm:measureHeavyWigner} to the heavy-Wigner tensor $T_N^B$. This yields convergence in probability of $\mu_{\Tilde{A^B}}$ to $\mu_B$, whose moments are $\eta_k^{(M'_r)}$ for $k\geq 1$, where
    $$M'_{r}=\frac{(p-2)!^{r/2} M_{r}}{ M_2^{r/2} }.$$
    Since $M_2=\frac{\alpha}{2-\alpha} B^{2-\alpha} $ and $M_{r} = \frac{\alpha}{k-\alpha} B^{r-s} (\theta + (-1)^r (1-\theta)) $, we find the desired result. 
    It has unbounded support because the sequence $(\eta_k)$ grows faster than any exponential sequence, see Lemma \ref{rem:no_compact_support}. It is uniquely determined by its moments by Lemma \ref{lem:caract_moments}.
    Finally, centering is a rank-one perturbation. Therefore, it does not affect the limiting ESD.
\end{proof}

\begin{lemma}\label{lem:tightness}
    The sequence of probability measures $(\E[\mu_{\Tilde{A}_N^B}])_{N\geq 1}$ is tight for the weak topology in $\mathcal{P}(\R)$. Consequently, $\mu^B$ converges weakly to some limiting measure $\mu$ as $B\rightarrow\infty$.
\end{lemma}
\begin{proof}
    The proof is similarto that of Lemma $3.1$ in \cite{ben_arous_guionnet}. First, recall from Remark \ref{rem:eta2} that
    \[
    \int x^2\,\mu^B(dx)=1
\]
for every $B>0$. By Markov's inequality, for every $R>0$, we have
\[
    \mu^B(|x|>R)
    \leq \frac{1}{R^2}\int x^2\,\mu^B(dx)
    =\frac{1}{R^2},
\]
which is uniform in $B$. Given $\varepsilon>0$, choose
$R>\varepsilon^{-1/2}$. Then
\[
    \sup_{B>0}\mu^B([-R,R]^c)\leq \varepsilon,
\]
which proves tightness of the family $(\mu^B)_{B>0}$. Since $\mathbb R$ is a Polish space, Prokhorov's theorem implies that this tight family is relatively compact for weak convergence. Consequently every sequence $B_j\to\infty$ admits a subsequence along which $\mu^{B_j}$ converges weakly. Now write $T_N=a_{N^{p-1}}^{-1} X_N$ for the untruncated tensor and $T_N^{B,0}$ for the non-centered truncated tensor. Its centered version is $T_N^B=T_N^{B,0}-\E T_N^{B,0}$. Let
\[
    A_N=N^{(p-2)/2}T_N.u^{p-2},
    \qquad
    A_N^{B,0}=N^{(p-2)/2}T_N^{B,0}.u^{p-2},
    \qquad
    A_N^B=N^{(p-2)/2}T_N^B.u^{p-2}.
\]
A common deterministic scalar factor in all three matrices is irrelevant for the rank estimates below. We first recall the rank inequality, which is a consequence of the interlacing inequalities. If $M,M'$ are Hermitian $N\times N$ matrices, let $F_M,F_{M'}$ denote the distribution function of $\mu_M, \mu_{M'}$ respectively. We have
\[
    \sup_{x\in\mathbb R}|F_M(x)-F_{M'}(x)|
    \leq
    \frac{1}{N}\mathrm{rank}(M-M').
\]
Define $\overline\mu_N:=\E[\mu_{A_N}], \overline\mu_N^B:=\E[\mu_{A_N^B}]$. Taking expectations in the rank inequality gives
\[
    d_L(\overline\mu_N,\overline\mu_N^B)
    \leq
    \frac{1}{N}\E[\mathrm{rank}(A_N-A_N^B)],
\]
where $d_L$ is the Lévy distance. We now estimate this rank. The difference between the untruncated tensor and the non-centered truncated tensor is supported exactly on the entries satisfying
$|x_{i_1,\ldots,i_p}|>B a_{N^{p-1}}$. Hence
\[
    A_N-A_N^{B,0}
    =
    \sum_{i_1\leq \ldots \leq i_p}
    \frac{x_{i_1,\ldots,i_p}}{a_{N^{p-1}}}
    \mathbf 1_{\{|x_{i_1,\ldots,i_p}|>B a_{N^{p-1}}\}} C_{i_1,\ldots,i_p},
\]
where $C_{i_1,\ldots,i_p}$ is the deterministic matrix obtained by contracting the symmetrized elementary tensor associated with $(i_1,\ldots,i_p)$. For each $(i_1,\ldots,i_p)$, the matrix $C_{i_1,\ldots,i_p}$ has non-zero rows and columns only in $\{i_1,\ldots,i_p\}$. Therefore
\[
    \mathrm{rank}(C_{i_1,\ldots,i_p})\leq |\{i_1,\ldots,i_p\}|\leq p.
\]
Let $L_N^B:= \sum_{i_1\leq \ldots \leq i_p} \mathbf 1_{\{|x_{i_1,\ldots,i_p}|>B a_{N^{p-1}}\}}$ be the number of tensor entries exceeding the truncation threshold.
By subadditivity of the rank,
\[
    \mathrm{rank}(A_N-A_N^{B,0})\leq \sum_{i_1\leq \ldots \leq i_p}
    \mathbf 1_{\{|x_{i_1,\ldots,i_p}|>B a_{N^{p-1}}\}} \mathrm{rank}(C_{i_1,\ldots,i_p}) \leq pL_N^B.
\]
It remains only to account for the centering. Since the entries of $T_N^{B,0}$ have the same expectation, the tensor $\E T_N^{B,0}$ is constant, and its contraction is a constant matrix. Therefore
\[
    \mathrm{rank}(A_N^{B,0}-A_N^B)\leq 1.
\]
Consequently, we have $\mathrm{rank}(A_N-A_N^B)\leq pL_N^B+1$. Taking expectations gives
\[
    \frac{1}{N}\E[\mathrm{rank}(A_N-A_N^B)]
    \leq
    \frac{p}{N}\E L_N^B+\frac1N.
\]
Now, since $\E L_N^B=\binom{N+p-1}{p} \P(|x_{\mathbf i}|>B a_{N^{p-1}})$, we get 
$$\frac{p}{N}\E L_N^B \sim \frac{1}{(p-1)!} N^{p-1}\P(|x_{\mathbf i}|>B a_{N^{p-1}}) \sim \frac{1}{(p-1)!}B^{-\alpha},$$ 
by regular variation and the definition of $a_{N^{p-1}}$. Therefore,
\[
    \limsup_{N\to\infty}
    d_L(\overline\mu_N,\overline\mu_N^B)
    \leq
    \frac{1}{(p-1)!}B^{-\alpha}.
\]
Now fix $B,C>0$. Since $\overline\mu_N^B\Rightarrow \mu^B$ and $\overline\mu_N^C\Rightarrow \mu^C$, the triangle inequality gives
\[
\begin{aligned}
    d_L(\mu^B,\mu^C)
    &\leq
    \limsup_{N\to\infty}
    \Big(
        d_L(\mu^B,\overline\mu_N^B)
        +
        d_L(\overline\mu_N^B,\overline\mu_N)
        +
        d_L(\overline\mu_N,\overline\mu_N^C)
        +
        d_L(\overline\mu_N^C,\mu^C)
    \Big) \\
    &\leq
    \frac{1}{(p-1)!}
    \left(B^{-\alpha}+C^{-\alpha}\right).
\end{aligned}
\]
Let now $B_j\to\infty$ and $C_j\to\infty$ be such that $\mu^{B_j}\Rightarrow \nu$ and $\mu^{C_j}\Rightarrow \nu'$. Since the Lévy distance metrizes weak convergence, $d_L(\mu^{B_j},\nu)\to 0$ and $d_L(\mu^{C_j},\nu')\to 0$.
Moreover,
\[
    d_L(\mu^{B_j},\mu^{C_j})
    \leq
    \frac{1}{(p-1)!}
    \left(B_j^{-\alpha}+C_j^{-\alpha}\right)
    \to 0,
\]
We finally obtain
\[
\begin{aligned}
    d_L(\nu,\nu')
    &\leq
    d_L(\nu,\mu^{B_j})
    +
    d_L(\mu^{B_j},\mu^{C_j})
    +
    d_L(\mu^{C_j},\nu')
    \to 0.
\end{aligned}
\]
Hence $\nu=\nu'$. Therefore every weakly convergent subsequence of $(\mu^B)_{B>0}$ has the same limit. Thus $(\mu^B)_{B>0}$ converges weakly to a measure denoted by $\mu_{\alpha,\theta}$. Indeed, suppose for contradiction that there exists $\varepsilon >0$ and a subsequence $\mu^{B_j}$, $B_j \to \infty$, such that for all $j$, $d_L(\mu^{B_j},\mu_{\alpha,\theta})>\varepsilon$. By tightness, we extract from $\mu^{B_j}$ a weakly converging subsequence, it must converge to $\mu_{\alpha,\theta}$, that is a contradiction. The result is proved.
\end{proof}

\begin{proof}[Proof of Theorem \ref{thm:Lévy}]
    The first point of the Theorem is proved by Corollary \ref{cor:convLévy}. The second point is given by Lemma \ref{lem:tightness}.
\end{proof}

\begin{remark}
    For the Lévy tensor, it has been proved by Bordenave and Chafaï in \cite{bordenave_chafai} that the limiting measure $\mu_\alpha$ is the spectral measure of the adjacency operator of the {\em Poisson Weighted Infinite Tree} (PWIT) introduced by Aldous in \cite{aldous_pwit}, with intensity 
     $$\Lambda_\alpha=\frac{\alpha}{2}x^{-\frac{\alpha}{2}-1}dx.$$ 
    We may also define a $p$-PWIT where the hyperedges attached to a vertex follow a Poisson point process. We conjecture that one may similarly obtain the local weak convergence for the truncated Lévy tensor. The remaining question is how to remove the truncation when we no longer have tightness arguments.
\end{remark}



\bibliographystyle{abbrv}
\bibliography{main}








\end{document}